\documentclass[a4paper]{amsart}
\usepackage{amssymb, enumitem}
\usepackage[all]{xy}
\usepackage{hyperref, aliascnt}

\newtheorem{lma}{Lemma}[section]

\newaliascnt{thmCt}{lma}
\newtheorem{thm}[thmCt]{Theorem}
\aliascntresetthe{thmCt}

\newaliascnt{corCt}{lma}
\newtheorem{cor}[corCt]{Corollary}
\aliascntresetthe{corCt}

\newaliascnt{prpCt}{lma}
\newtheorem{prp}[prpCt]{Proposition}
\aliascntresetthe{prpCt}

\theoremstyle{definition}

\newaliascnt{dfnCt}{lma}
\newtheorem{dfn}[dfnCt]{Definition}
\aliascntresetthe{dfnCt}

\newaliascnt{rmkCt}{lma}
\newtheorem{rmk}[rmkCt]{Remark}
\aliascntresetthe{rmkCt}

\newaliascnt{exaCt}{lma}
\newtheorem{exa}[exaCt]{Example}
\aliascntresetthe{exaCt}

\newaliascnt{qstCt}{lma}
\newtheorem{qst}[qstCt]{Question}
\aliascntresetthe{qstCt}

\newaliascnt{pbmCt}{lma}
\newtheorem{pbm}[pbmCt]{Problem}
\aliascntresetthe{pbmCt}

\newaliascnt{ntnCt}{lma}
\newtheorem{ntn}[ntnCt]{Notation}
\aliascntresetthe{ntnCt}

\def\today{\number\day\space\ifcase\month\or   January\or February\or
   March\or April\or May\or June\or   July\or August\or September\or
   October\or November\or December\fi\   \number\year}

\newcommand{\ZZ}{{\mathbb{Z}}}
\newcommand{\NN}{{\mathbb{N}}}
\newcommand{\CC}{{\mathbb{C}}}
\newcommand{\RR}{{\mathbb{R}}}
\newcommand{\QQ}{{\mathbb{Q}}}
\newcommand{\id}{{\mathrm{id}}}
\newcommand{\andSep}{\,\,\,\text{ and }\,\,\,}

\DeclareMathOperator{\linSpan}{span}
\DeclareMathOperator{\trace}{trace}

\DeclareMathOperator{\kernel}{ker}
\DeclareMathOperator{\nil}{nil}
\DeclareMathOperator{\rad}{rad}
\DeclareMathOperator{\Idem}{Idem}

\newcommand{\ca}{$C^*$-algebra}

\newcounter{theoremintro}

\newaliascnt{dfnIntroCt}{theoremintro}
\newtheorem{dfnIntro}[dfnIntroCt]{Definition}
\aliascntresetthe{dfnIntroCt}

\newtheorem{corIntro}[theoremintro]{Corollary}
\newtheorem{qstIntro}[theoremintro]{Question}

\newaliascnt{thmIntroCt}{theoremintro}
\newtheorem{thmIntro}[thmIntroCt]{Theorem}
\aliascntresetthe{thmIntroCt}

\title{Zero-product balanced algebras}
\date{\today}

\author[Eusebio Gardella]{Eusebio Gardella}
\address{Eusebio Gardella,
Department of Mathematical Sciences, Chalmers University of
Technology and University of Gothenburg, Gothenburg SE-412 96, Sweden.}
\email{gardella@chalmers.se}
\urladdr{www.math.chalmers.se/~gardella}

\author{Hannes Thiel}
\address{Hannes~Thiel, 
Department of Mathematical Sciences, Chalmers University of Technology and University of
Gothenburg, Gothenburg SE-412 96, Sweden.}
\email{hannes.thiel@chalmers.se}
\urladdr{www.hannesthiel.org}

\thanks{
The first named author was partially supported by the Deutsche Forschungsgemeinschaft (DFG, German Research Foundation) through an Eigene Stelle, and by the Swedish Research Council Grant 2021-04561.
The second named author was partially supported by the ERC Consolidator Grant No.~681207 and by the Knut and Alice Wallenberg Foundation (KAW 2021.0140).
}

\subjclass[2010]%
{Primary
15A86, 
47B49; 
Secondary
16N40, 
16S50, 
16U99, 
47B47. 
}

\date{\today}

\begin{document}

\begin{abstract}
We say that an algebra is \emph{zero-product balanced} if $ab\otimes c$ and $a\otimes bc$ agree modulo tensors of elements with zero-product.
This is closely related to but more general than the notion of a \emph{zero-product determined} algebra introduced and developed by Bre\v{s}ar, Villena and others.
Every surjective, zero-product preserving map from a zero-product balanced algebra is automatically a weighted epimorphism, and this implies that zero-product balanced algebras are determined by their linear and zero-product structure.
Further, the commutator subspace of a zero-product balanced algebra can be described in terms of square-zero elements.

We show that a commutative, reduced algebra is zero-product balanced if and only if it is generated by idempotents.
It follows that every commutative, zero-product balanced algebra is spanned by nilpotent and idempotent elements.
\end{abstract}

\maketitle

\section{Introduction}

A linear map $\pi \colon A \to B$ between algebras is said to \emph{preserve zero-products} if $ab = 0$ implies $\pi(a)\pi(b) = 0$, for all $a,b \in A$.
The typical example of such a map is a \emph{weighted homomorphism}, namely the composition of an algebra homomorphism $A \to B$ with a centralizer on $B$, and a much studied problem is to determine situations in which zero-product preserving, linear maps are automatically of this form; see, for example, \cite{CheKeLeeWon03PresZeroProd, CheKeLee04MapsCharZeroProd, BreGraOrt09ZPDMatrixAlg, AlaBreExtVil09MapsPresZP, Bre21BookZeroProdDetermined}.

Maps preserving zero-products or more general forms of orthogonality occur naturally in many different settings. For instance,
Banach showed in 1932 that for $p\neq 2$, every linear, isometric map on $L^p([0,1])$ sends functions with disjoint support to functions with disjoint support, and Lamperti showed in 1958 that this characterizes linear isometries on arbitrary $L^p$-spaces.
This was later generalized to maps between other Banach lattices; 
we refer to the monograph \cite{AbrKit00InvDisjPreserving}.
The intimate connection between being isometric and preservation of orthogonality is also shown by the results of Koldobsky \cite{Kol93OpsPresOrthoAreIso} and Blanco-Turn\v{s}ek \cite{BlaTur06PresOrthoNormedSp} that a linear map between normed spaces preserves orthogonality in the Birkhoﬀ–James sense if and only if it is a scalar multiple of a linear isometry.

In the structure theory of \ca{s}, zero-product preserving, positive maps \cite{WinZac09CpOrd0, GarThi22arX:WeightedHomo} are used to model noncommutative partitions of unity, leading to the notion of nuclear dimension \cite{WinZac10NuclDim} -- a noncommutative covering dimension -- which plays a central role in the classification theory of nuclear \ca{s}.

\smallskip

Back to the basic problem:
When is a zero-product preserving, linear map $\pi \colon A \to B$ between algebras automatically a weighted homomorphism?
Since the range of a weighted homomorphism is essentially a subalgebra, it is natural to restrict attention to the case that $\pi$ is \emph{surjective}.

To tackle the basic problem, we focus on a property that lies between preservation of zero-products and being a weighted homomorphism:

\begin{dfnIntro}
\label{dfn:Semimult}
We say that a map $\pi \colon A \to B$ between algebras is \emph{semimultiplicative} if
$\pi(ab)\pi(c) = \pi(a)\pi(bc)$ for all $a,b,c \in A$.
\end{dfnIntro}

This property has implicitly appeared in the literature before, but as far as we know it was never given a name and it has not been systematically studied so far.

It is easy to see that every weighted homomorphism is semimultiplicative, and we show in \autoref{prp:WeightedEpi} that the converse holds under the mild technical assumption that $A$ is idempotent and that $B$ is idempotent and faithful (which includes the case that $A$ is unital, and $B$ is unital, or simple, or a Banach algebra with bounded approximate identity).
It is also easy to see that every surjective, semimultiplicative map from an idempotent algebra to a faithful algebra preserves zero-products.
Thus, the following question arises naturally:

\begin{center}
When is a zero-product preserving map semimultiplicative?
\end{center} 

The main concept of this paper is a property that captures exactly when the above question has a positive answer.
Throughout this paper, $K$ denotes a unital, commutative ring.

\begin{dfnIntro}[\ref{dfn:Balanced}]
We say that a $K$-algebra $A$ is \emph{zero-product balanced} if
\[
ab\otimes c - a\otimes bc \in \linSpan_K \big\{ u\otimes v\in A \otimes_K A : uv=0 \big\}
\]
for all $a,b,c \in A$.
\end{dfnIntro}

Zero-product balancedness is closely related to the concept of a \emph{zero-product determined} algebra introduced by Bre\v{s}ar, Gra\v{s}i\v{c} and Ortega in \cite{BreGraOrt09ZPDMatrixAlg}
It has also been extensively studied in the context of Banach algebras under the name `property~$\mathbb{B}$', \cite{AlaBreExtVil09MapsPresZP}, and under the name `algebraic property~$\mathbb{B}$' in \cite{AlaExtVilBreSpe16CommutatorsSquareZero}.
For details on these concepts, we refer to the recent book by Bre\v{s}ar, \cite{Bre21BookZeroProdDetermined}.

We show that every zero-product balanced algebra is zero-product determined (\autoref{prp:ZPDImplBalanced}), and that the converse holds for algebras admitting a certain factorization (\autoref{prp:MultipleFactorization}), which includes unital algebras as well as Banach algebras with a bounded left approximate identity (\autoref{prp:BanachAlgBalanced}).
However, \autoref{exa:BalancedNotZPD} shows that there are algebras that are zero-product balanced but not zero-product determined, even commutative and 
finite-dimensional ones.

The following result summarizes our findings:

\begin{thmIntro}[\ref{prp:WeightedEpi}, \ref{prp:WeightedFromBalanced}]
Let $\pi \colon A \to B$ be a surjective, linear map between idempotent algebras.
Assume that $B$ is faithful.
Consider the following properties:
\begin{enumerate}
\item
$\pi$ preserves zero-products;
\item
$\pi$ is semimultiplicative;
\item
$\pi$ is a weighted homomorphism.
\end{enumerate}

Then (1)$\Leftarrow$(2)$\Leftrightarrow$(3).
If $A$ is zero-product balanced, then (1)$\Rightarrow$(2) as well, 
and then (1)--(3) are equivalent.
\end{thmIntro}

An important consequence is that zero-product balanced algebras are determined by their linear and zero-product structure:

\begin{corIntro}[\ref{prp:IsoDetByZPStructure}]
Let $A$ be a zero-product balanced, idempotent algebra, and let~$B$ be a faithful, idempotent algebra.
Then $A$ and $B$ are isomorphic if and only if they admit a bijective, zero-product preserving, linear map $A \to B$. 
\end{corIntro}

It is a well-studied problem to determine in which algebras every additive commutator $[a,b] := ab - ba$ can be expressed as a sum of square-zero elements \cite{WanWu91SumsSquareZero, CheLeePuc10CommNilSimpleRings, AlaExtVilBreSpe16CommutatorsSquareZero}.
In \autoref{sec:FN2}, we show that this is always possible in zero-product balanced algebras, and we even give a precise description of the commutator subspace in terms of special square-zero elements.
We define an \emph{orthogonally factorizable square-zero element} as an element $x$ such that there exist $y$ and $z$ with $x = yz$ and $zy = 0$.
We use $FN_2(A)$ to denote the collection of such elements in an algebra $A$;
see \autoref{dfn:FN2}.

\begin{thmIntro}[\ref{prp:BalancedCommutatorFN2}]
Let $A$ be a zero-product balanced, idempotent $K$-algebra.
Then the subspace of $A$ generated by additive commutators agrees with the subspace generated by $FN_2(A)$.
\end{thmIntro}

The last three sections of the paper are devoted to the following question:

\begin{qstIntro}
Which algebras are zero-product balanced?
\end{qstIntro} 

If an algebra is generated by idempotents, then it is zero-product balanced;
see \autoref{prp:IdemGenBalanced}.
In certain situations, the converse also holds:
In \cite[Theorem~3.7]{Bre16FinDimZeroProdDet}, Bre\v{s}ar shows that a unital, finite-dimensionl $K$-algebra over a field $K$ is zero-product balanced (equivalently, zero-product determined) if and only if it is generated by idempotents.
The following result is in the same spirit:

\begin{thmIntro}[\ref{prp:SemiprimeCommutative}]
A commutative $K$-algebra over a field $K$ is reduced and zero-product balanced if and only if it is generated by idempotents.

In this case, the algebra is isomorphic to the algebra of finite-valued, continuous functions $X \to K$ with compact support for some Boolean space $X$.
\end{thmIntro}

We deduce some structure results in the non-reduced case:
Let $A$ be a commutative, zero-product balanced $K$-algebra over a field $K$.
Then the Jacobson radical agrees with the prime radical (\autoref{prp:Radicals}), every element in $A$ is a sum of a nilpotent element and a linear combination of finitely many pairwise orthogonal idempotents (\autoref{prp:CommutativeSpanned}), and $A$ is a semiregular, clean, exchange ring (\autoref{prp:Semiregular}).
This leads to a dichotomy:
A commutative, zero-product balanced algebra either admits a character (a homomorphism to the base field) or is nilradical;
see \autoref{prp:DichotomyCommutative}.
Using this, we obtain a general dichotomy result for zero-product balanced algebras:

\begin{thmIntro}[\ref{prp:DichotomyGeneral}]
Let $A$ be a zero-product balanced $K$-algebra over a field~$K$.
Then either $A$ has a character, or $A$ is a radical extension over its commutator ideal, that is, for every $a\in A$ exist $m,n\geq 1$ and $r_j,x_j,y_j,s_j\in A$ such that
\[
a^m = \sum_{j=1}^n r_j [x_j,y_j] s_j.
\]
\end{thmIntro}

\subsection*{Conventions}

All rings are associative, but possibly nonunital and noncom\-mu\-ta\-tive.
Ideal means two-sided ideal.

\subsection*{Acknowledgements}

The authors thank Matej Bre\v{s}ar and Mikhail Chebotar for valuable comments on an earlier version of this paper.
The authors also thank the anonymous referee for their useful feedback.

\section{Zero-product balanced algebras}
\label{sec:Balanced}

In this section, after recalling the definition of a zero-product determined algebra from \cite{BreGraOrt09ZPDMatrixAlg}, we introduce the more general notion of a \emph{zero-product balanced} algebra;
see \autoref{dfn:Balanced}.
Every zero-product determined algebra is also zero-product balanced; 
see \autoref{prp:ZPDImplBalanced}.
We show that the converse holds for algebras admitting a certain factorization (\autoref{prp:MultipleFactorization}), which includes unital algebras as well as Banach algebras with a bounded left approximate identity (\autoref{prp:BanachAlgBalanced}).
In particular, a \ca{} is zero-product determined if and only if it is zero-product balanced.
\autoref{exa:BalancedNotZPD} shows that there exist zero-product balanced algebras that are not zero-product determined.

We also observe that zero-product balancedness passes to (certain) ideals and quotients;
see Propositions~\ref{prp:BalancedIdeal} and~\ref{prp:BalancedQuotient}.
In particular, an algebra is zero-product balanced whenever its multiplier algebra is, and we use this in later sections to generalize some results in the literature to the nonunital setting;
see \autoref{prp:MatrixAlgBalanced} and \autoref{exa:MatrixAlgIso}.

\smallskip
Several results in this section, or versions thereof, have appeared in \cite{Bre21BookZeroProdDetermined}.
In particular, versions of Propositions~\ref{prp:CharZPD} and \ref{prp:CharZPD-Z2} are given in \cite[Proposition~1.3]{Bre21BookZeroProdDetermined}, which also implies \autoref{prp:ZPDImplBalanced}.
Further, in the context of Banach algebras, a version of \autoref{prp:CharBalanced} has appeared as \cite[Proposition~5.2]{Bre21BookZeroProdDetermined}, and versions of 
Propositions~\ref{prp:BalancedIdeal} and~\ref{prp:BalancedQuotient} have appeared as \cite[Theorem~5.8]{Bre21BookZeroProdDetermined}.

\smallskip

Throughout this section, $K$ denotes a commutative, unital ring.
Given a subset~$X$ of a $K$-module $V$, we use
\[
\linSpan_K X := \big\{ \lambda_1 x_1+\ldots+\lambda_n x_n \in V\colon n\geq 1, \lambda_1,\ldots,\lambda_n\in K, x_1,\ldots,x_n\in X \big\}
\]
to denote the $K$-submodule of $V$ generated by $X$.

\begin{dfn}
\label{dfn:ZPD}
Given a $K$-module $V$, a $K$-algebra $A$, and a $K$-bilinear map $\varphi\colon A\times A\to V$, one says that $\varphi$ \emph{preserves zero-products} if for all $a,b\in A$ with $ab=0$, we have $\varphi(a,b)=0$.

Following \cite{BreGraOrt09ZPDMatrixAlg}, we say that $A$ is \emph{zero-product determined} if for every $K$-module $V$ and every $K$-bilinear map $\varphi\colon A\times A\to V$ that preserves zero-products, there exists a $K$-linear map $\Phi\colon \linSpan_K A^2\to V$ such that $\varphi(a,b)=\Phi(ab)$ for all $a,b\in A$, where $\linSpan_K A^2$ denotes the submodule of $A$ generated by $\{ab:a,b\in A\}$.
\end{dfn}

The next result is contained in \cite[Proposition~1.3]{Bre21BookZeroProdDetermined}.
(The assumption that $K$ is a field is not necessary for the equivalence of~(i) and~(ii) of \cite[Proposition~1.3]{Bre21BookZeroProdDetermined}.)

\begin{prp}
\label{prp:CharZPD}
A $K$-algebra $A$ is zero-product determined if and only if for every zero-product preserving $K$-bilinear map $\varphi\colon A\times A\to V$ to some $K$-module $V$, we have $\sum_{j=1}^n\varphi(a_j,b_j)=0$ for all $a_1,\ldots,a_n,b_1,\ldots,b_n\in A$ satisfying $\sum_{j=1}^n a_jb_j=0$.
\end{prp}

We will repeatedly use the following notation.

\begin{ntn}
\label{ntn:Z2A}
Let $A$ be a $K$-algebra. 
We set 
\[
Z_2(A) = \linSpan_K \big\{ u\otimes v\in A\otimes_K A : uv=0 \big\}.
\]
\end{ntn}

\begin{prp}
\label{prp:CharZPD-Z2}
Let $A$ be a $K$-algebra, and let $m \colon A \otimes_K A \to A$ be the linear map satisfying $m(a \otimes b) = ab$ for $a,b \in A$.
Then $A$ is zero-product determined if and only if $\kernel(m) \subseteq Z_2(A)$.
\end{prp}
\begin{proof}
This follows from \autoref{prp:CharZPD} applied to the universal zero-product preserving, bilinear map $\varphi \colon A \times A \to (A \otimes_K A)/Z_2(A)$ given by
\[
\varphi(a,b) = (a \otimes b) + Z_2(A)
\]
for $a,b \in A$.
\end{proof}

\begin{exa}
\label{exa:ZeroProduct}
Every algebra $A$ with zero-product (that is, $ab = 0$ for all $a,b \in A$) is zero-product determined.
Indeed, in this case we have $\kernel(m) = A \otimes_K A = Z_2(A)$.
\end{exa}

\begin{dfn}
\label{dfn:Balanced}
We say that a $K$-algebra $A$ is \emph{zero-product balanced} if
\[
ab \otimes c - a \otimes bc \in  Z_2(A)
\]
for all $a,b,c \in A$.
\end{dfn}

\begin{exa}
\label{exa:ZeroTripleProduct}
Let $A$ be an algebra such that $A^3=\{0\}$, that is, $abc = 0$ for all $a,b,c \in A$.
Then $A$ is zero-product balanced, since already $ab \otimes c$ and $a \otimes bc$ themselves belong to $Z_2(A)$ for all $a,b,c \in A$.

On the other hand, in \autoref{exa:BalancedNotZPD} we give examples of algebras with $A^3=\{0\}$ such that $A$ is not zero-product determined.
\end{exa}

\begin{prp}
\label{prp:CharBalanced}
Let $A$ be a $K$-algebra.
Then the following are equivalent:
\begin{enumerate}
\item
$A$ is zero-product balanced.
\item
For every $K$-module $V$ and every $K$-bilinear map $\varphi\colon A\times A\to V$ that preserves zero-products we have $\varphi(ab,c)=\varphi(a,bc)$ for all $a,b,c\in A$.
\end{enumerate}
\end{prp}
\begin{proof}
Let us show that~(1) implies~(2).
Assuming~(1), let $\varphi\colon A\times A\to V$ be a bilinear map to a $K$-module $V$ that preserves zero-products.
By the universal property of the tensor product, this induces a linear map $\bar{\varphi}\colon A\otimes_K A\to V$ such that $\varphi(a,b)=\bar{\varphi}(a\otimes b)$ for all $a,b\in A$.
Given $u,v\in A$ with $uv=0$, we have
\[
\bar{\varphi}(u\otimes v)=\varphi(u,v)=0,
\]
and it follows that $\bar{\varphi}$ vanishes on $Z_2(A)$.
Thus, given $a,b,c\in A$, using that $ab\otimes c - a\otimes bc\in Z_2(A)$ at the second step, we get
\[
\varphi(ab,c)
= \bar{\varphi}(ab\otimes c)
= \bar{\varphi}(a\otimes bc)
= \varphi(a,bc).
\]

The situation is shown in the following commutative diagram:

\[
\xymatrix{
A\times A \ar[r] \ar[dr]_{\varphi}
& A\otimes_K A \ar[r] \ar[d]^{\bar{\varphi}}
& (A\otimes_K A) / Z_2(A), \ar@{-->}[dl] \\
& V
}
\]

Conversely, assume (2) and consider the universal bilinear map $\varphi\colon A\times A\to A\otimes_K A$ given by $\varphi(a,b)=a\otimes b$.
Let $(A\otimes_K A)/ Z_2(A)$ be the quotient as $K$-modules, and let $\pi\colon A\otimes_K A \to (A\otimes_K A)/ Z_2(A)$ denote the quotient map.
Then $\pi\circ\varphi$ is a bilinear map that preserves zero-products.
Hence, given $a,b,c\in A$, using the assumption at the second step, we get
\[
\pi(ab\otimes c)
= (\pi\circ\varphi)(ab,c)
= (\pi\circ\varphi)(a,bc)
= \pi(a\otimes bc)
\]
and thus $ab\otimes c - a\otimes bc \in Z_2(A)$, as desired.
\end{proof}

Recall that throughout this section, $K$ denotes a commutative, unital ring.

\begin{prp}
\label{prp:BalancedChangeCoefficients}
Let $L\subseteq K$ be a unital subring (with the same unit), and let $A$ be a $K$-algebra.
Assume that $A$ is zero-product determined (zero-product balanced) as an $L$-algebra.
Then $A$ is zero-product determined (zero-product balanced) as a $K$-algebra.
\end{prp}
\begin{proof}
We use the characterization of zero-product balancedness from \autoref{prp:CharBalanced}.
Assume that $A$ is zero-product balanced as a $L$-algebra.
Let $\varphi\colon A\times A\to V$ be a $K$-bilinear map to some $K$-module $V$.
Considering $V$ as a $L$-module, the map $\varphi$ is $L$-bilinear, and we obtain by assumption that $\varphi(ab,c)=\varphi(a,bc)$ for all $a,b,c\in A$.

The proof for zero-product determination is similar, using \autoref{prp:CharZPD}.
\end{proof}

In particular, if a $K$-algebra $A$ is zero-product balanced as a ring, then it is also zero-product balanced as a $K$-algebra.
Thus, even when considering a Banach algebra or a \ca{}, it is most interesting to determine if it is zero-product balanced as a ring.

\begin{prp}
\label{prp:ZPDImplBalanced}
Every zero-product determined algebra is zero-product balanced.
\end{prp}
\begin{proof}
Let $A$ be a zero-product determined $K$-algebra.
To verify \autoref{prp:CharBalanced}(2), let $\varphi\colon A\times A\to V$ be a $K$-bilinear map to some $K$-module $V$.
By zero-product determination, there exists a linear map $\Phi\colon\linSpan_K A^2\to V$ such that
$\varphi(a,b)=\Phi(ab)$ for all $a,b\in A$.
It follows that
\[
\varphi(ab,c)=\Phi(abc)=\varphi(a,bc)
\]
for all $a,b,c\in A$.
\end{proof}

If $A$ is unital, then the converse also holds:
If $A$ is zero-product balanced, and $\varphi\colon A\times A\to V$ is a bilinear map, then $\varphi(ab,c)=\varphi(a,bc)$ for all $a,b,c\in A$, which implies that the map $\Phi\colon A\to V$ given by $\Phi(a):=\varphi(a,1)$ satisfies
\[
\varphi(a,b) = \varphi(ab,1) = \Phi(ab)
\]
for all $a,b\in A$. However, and unlike in \cite{Bre16FinDimZeroProdDet}, we will be particularly interested in nonunital algebras. 

The next result shows that the converse to \autoref{prp:ZPDImplBalanced} also holds if $A$ admits a certain factorization (which is automatic in the unital case).

\begin{prp}
\label{prp:MultipleFactorization}
Let $A$ be a $K$-algebra such that for every $n\geq 1$ and $a_1,\ldots,a_n\in A$ there exist $x,y_1,\ldots,y_n\in A$ such that $a_j=xy_j$ for $j=1,\ldots,n$.
Then $A$ is zero-product determined if and only if $A$ is zero-product balanced.
\end{prp}
\begin{proof}
Assume that $A$ is zero-product balanced, and let $\varphi\colon A\times A\to V$ be a zero-product preserving $K$-bilinear map to some $K$-module $V$.
To apply \autoref{prp:CharZPD}, let $a_j,b_j\in A$ for $j=1,\ldots,n$ such that $\sum_{j=1}^n a_jb_j=0$.
We need to show that $\sum_{j=1}^n \varphi(a_j,b_j)=0$.
By assumption, there exist $x,y_1,\ldots,y_n\in A$ such that $a_j=xy_j$ for $j=1,\ldots,n$.
Then
\[
x\Big(\sum_{j=1}^n y_jb_j\Big) = \sum_{j=1}^n a_jb_j = 0.
\]
Using that $A$ is zero-product balanced at the second step, and using that~$\varphi$ preserves zero-products at the last step, we obtain 
\[
\sum_{j=1}^n \varphi(a_j,b_j)
= \sum_{j=1}^n \varphi(xy_j,b_j)
= \sum_{j=1}^n \varphi(x,y_jb_j)
= \varphi\Big(x,\sum_{j=1}^n y_jb_j\Big)=0,
\]
as desired.
\end{proof}

\begin{cor}
\label{prp:BanachAlgBalanced}
Let $A$ be a Banach algebra with a bounded left approximate identity.
Then $A$ is zero-product determined as a $\CC$-algebra (as a ring) if and only if $A$ is zero-product balanced as a $\CC$-algebra (as a ring).
\end{cor}
\begin{proof}
Given $n\geq 1$ and $a_1,\ldots,a_n\in A$, the multiple Cohen factorization theorem (\cite[Theorem~17.1]{DorWic79FactorizationBanachMod}) provides $x,y_1,\ldots,y_n\in A$ such that $a_j=xy_j$ for $j=1,\ldots,n$.
This verifies the assumption of \autoref{prp:MultipleFactorization}, proving the result.
\end{proof}

\begin{rmk}
Using a right-sided version of \autoref{prp:MultipleFactorization} and a right-sided version of the multiple Cohen factorization theorem, we see that \autoref{prp:BanachAlgBalanced} also holds for Banach algebras with a bounded right approximate identity.
\end{rmk}

\begin{prp}
\label{prp:BalancedIdeal}
Let $A$ be a zero-product balanced $K$-algebra, and let $I\subseteq A$ be a (two-sided) ideal.
Assume that $I=\linSpan_K I\cdot A=\linSpan_K A\cdot I$.
(For example, this is the case if $A$ is unital.)
Then $I$ is a zero-product balanced $K$-algebra.
\end{prp}
\begin{proof}
The proof is similar to that of \cite[Proposition~2.7]{Bre16FinDimZeroProdDet}.
Let $\varphi\colon I\times I\to V$ be a zero-product preserving $K$-bilinear map to some $K$-module $V$.
We need to show that $\varphi(xy,z)=\varphi(x,yz)$ for all $x,y,z\in I$.

Given $x,y\in I$, consider the map $\psi_{x,y}\colon A\times A\to V$ given by $\psi_{x,y}(a,b):=\varphi(xa,by)$.
Note that $\psi_{x,y}$ is $K$-bilinear and zero-product preserving.
Using that $A$ is zero-product balanced, we obtain
\[
\varphi(xab,cy)
= \psi_{x,y}(ab,c)
= \psi_{x,y}(a,bc)
= \varphi(xa,bcy),
\]
for all $a,b,c\in A$.
Now let $x,y,z\in I$.
Using that $A=\linSpan_K I\cdot A=\linSpan_K A\cdot I$, choose $x_j,z_k\in I$ and $a_j,c_k\in A$ such that
\[
x = \sum_{j=1}^m x_ja_j \andSep
z = \sum_{k=1}^n c_kz_k.
\]
Then
\[
\varphi(xy,z)
= \sum_{j,k} \varphi(x_ja_jy, c_kz_k)
= \sum_{j,k} \varphi(x_ja_j, yc_kz_k)
= \varphi(x,yz),
\]
as desired.
\end{proof}

\begin{prp}
\label{prp:BalancedQuotient}
Let $A$ be a zero-product balanced $K$-algebra, and let $I\subseteq A$ be an ideal.
Then $A/I$ is a zero-product balanced $K$-algebra.
\end{prp}
\begin{proof}
Set $B:=A/I$, and let $\pi\colon A\to B$ be the quotient map.
Let $\varphi\colon B\times B\to V$ be a zero-product preserving $K$-bilinear map to some $K$-module $V$.
We need to show that $\varphi(xy,z)=\varphi(x,yz)$ for all $x,y,z\in B$.

Define $\psi\colon A\times A\to V$ by $\psi(a,b):=\varphi(\pi(a),\pi(b))$ for $a,b\in A$.
If $a,b\in A$ satisfy $ab=0$, then $\pi(a)\pi(b)=\pi(ab)=0$, and then $\psi(a,b)=0$, which shows that~$\psi$ preserves zero-products.
By assumption, we obtain $\psi(ab,c)=\psi(a,bc)$ for all $a,b,c\in A$.

Given $x,y,z\in B$, choose $a,b,c\in A$ with $\pi(a)=x$, $\pi(b)=y$ and $\pi(c)=z$.
Then
\[
\varphi(xy,z)
= \varphi(\pi(ab),\pi(c))
= \psi(ab,c)
= \psi(a,bc)
= \varphi(\pi(a),\pi(bc))
= \varphi(x,yz). \qedhere
\]
\end{proof}

We close this section with some explicit examples of zero-product balanced algebras, including some which are not zero-product determined (\autoref{exa:BalancedNotZPD}).
These examples are closely related to Examples~5.3 and~5.4 in \cite{Bre21BookZeroProdDetermined}, which are in the context of Banach algebras.

\begin{exa}
\label{exa:Nm}
Let $K$ be a field, and let $m \geq 2$. 
Let $N_m$ denote the universal $K$-algebra generated by a nilpotent element of order $m$. Note that $N_m$ is isomorphic to $N_m := xK[x]/\langle x^m \rangle$, and that it can be identified with the subalgebra of the matrix algebra $M_m(K)$ generated by the unilateral shift
\[
\begin{pmatrix}
0 & 1 & 0 & & 0 \\
0 & 0 & 1 & & 0 \\
\vdots & & \ddots & \ddots \\
0 & & \cdots & 0 & 1 \\
0 & & \cdots & 0 & 0
\end{pmatrix}
\]

We claim that $N_2$ and $N_3$ are zero-product determined (and thus also zero-product balanced by \autoref{prp:ZPDImplBalanced}).
Indeed, $N_2$ is isomorphic to $K$ with zero multiplication and therefore is zero-product determined by \autoref{exa:ZeroProduct}.
We have $N_3 = \{ \lambda x + \mu x^2 \colon \lambda, \mu \in K \}$.
Given any $K$-vector space $V$ and a zero-product preserving, bilinear map $\varphi \colon N_3 \times N_3 \to V$, we have
\[
\varphi( \lambda x + \mu x^2, \sigma x + \tau x^2 )
= \varphi(\lambda x, \sigma x)
= \varphi(\lambda\sigma x, x).
\]
With $\Phi \colon \linSpan_K N_3^2 = \{ \mu x^2 : \mu \in K\} \to V$ given by $\Phi(\mu x^2) = \varphi(\mu x, x)$, it follows that $\varphi(a,b)=\Phi(ab)$. 
Thus $N_3$ is zero-product determined. 
\vspace{0.2cm}

On the other hand, we claim that $N_m$ is not zero-product balanced for $m\geq 4$.
To prove this, we will show that $x^2 \otimes x - x \otimes x^2$ does not belong to $Z_2(N_m)$.

First, note that
\[
N_m = \big\{ \lambda_1 x^1 + \ldots + \lambda_{m-1} x^{m-1} : \lambda_1, \ldots, \lambda_{m-1} \in K \big\}.
\]
Given $a = \sum_j \lambda_j x^j \in N_m$, we set
\[
n(a) := \min \{ j : \lambda_j \neq 0 \}.
\]

Given $a,b \in N_m$, we have $ab=0$ if and only if $n(a)+n(b)\geq m$.
Let $\varphi \colon N_m \to K$ be the linear map satisfying $\varphi(x) = 1$ and $\varphi(x^j) = 0$ for $j \geq 2$, and set 
$\psi = \id \otimes \varphi \colon N_m \otimes_K N_m \to N_m \otimes_K K \cong N_m$. Then 
\[
\psi( b \otimes x ) = b \andSep
\psi( b \otimes x^j ) = 0
\]
for $b \in N_m$ and $j \geq 2$.
Given $a \in N_m$, we have $\varphi(a) \neq 0$ if and only if $n(a)=1$. 
Therefore, if $a,b \in N_m$ satisfy $ab=0$, then $\psi( a \otimes b ) \neq 0$ if and only if $n(a)=1$ and $n(b)=m-1$, that is, $b = \mu x^{m-1}$ for some $\mu \in K$, and then $\psi( a \otimes b ) = \varphi(a)\mu x^{m-1}$.

To reach a contradiction, assume that $x^2 \otimes x - x \otimes x^2 = \sum_j a_j \otimes b_j$ for some $a_j,b_j \in N_m$ with $a_j,b_j \in N_m$ and $a_jb_j = 0$.
As noted above, if $n(a_j)=1$, then $b_j = \mu_j x^{m-1}$ for some $\mu_j \in K$.
We have
\[
\psi( x^2 \otimes x - x \otimes x^2 ) = x^2.
\]
On the other hand, using that $m \geq 4$, we have
\[
\psi \big( \sum_j a_j \otimes b_j \big)
= \sum_{n(j)=1} \varphi(a_j)b_j
= \sum_{n(j)=1} \varphi(a_j)\mu_j x^{m-1} \neq x^2
\]
which is the desired contradiction.
\end{exa}

The following is an example of a zero-product balanced algebra which is not zero-product determined. 
By taking $K=\mathbb{R}$ and $D=\mathbb{C}$, one can construct the example to be a $4$-dimensional $\RR$-algebra.

\begin{exa}
\label{exa:BalancedNotZPD}
Let $K$ be a field, set $N_3 := xK[x]/\langle x^3 \rangle$ as in \autoref{exa:Nm}, and let~$D$ be any unital $K$-algebra without zero-divisors.
Set
\[
A := D \otimes_K N_3,
\]
with product induced by $(d \otimes a)(e \otimes b) := de \otimes ab$ for $d,e \in D$ and $a,b \in N_3$.

Note that the product of any three elements in $A$ is zero, whence $A$ is zero-product balanced by \autoref{exa:ZeroTripleProduct}.
Assume that $\dim_K(D) \geq 2$.
We claim that then~$A$ is not zero-product determined.
To see this, choose $d_0 \in D$ which is not a scalar multiple of the identity.
Consider the elements $d_0 \otimes x$ and $1 \otimes x$ in $A$.
Set
\[
t := (d_0 \otimes x) \otimes (1 \otimes x) - (1 \otimes x) \otimes (d_0 \otimes x) \in A \otimes_K A.
\]

Let $m \colon A \otimes_K A \to A$ be induced by the multiplication map $A \times A \to A$.
Then
\[
m(t)
= (d_0 \otimes x)(1 \otimes x) - (1 \otimes x)(d_0 \otimes x)
= (d_01 \otimes x^2) - (1d_0 \otimes x^2)
= 0.
\]
On the other hand, we will show that $t$ does not belong to $Z_2(A)$.
We have
\[
A = \big\{ d \otimes x + e \otimes x^2 : d,e \in D \big\}.
\]
Set $I := \{ e \otimes x^2 : e \in D \} \subseteq A$.
Let $\varphi \colon A \to D$ be the $K$-linear map satisfying $\varphi( d \otimes x ) = d$ and $\varphi( e \otimes x^2 ) = 0$ for $d,e \in D$,
and set $\psi := \varphi \otimes \varphi \colon A \otimes_K A \to D \otimes_K D$.

Given $a = d \otimes x + e \otimes x^2$ and $b = f \otimes x + g \otimes x^2$ in $A$, we have
\[
ab = df \otimes x^2.
\]
Using that $D$ has no zero divisors, it follows that $ab = 0$ if and only if $a \in I$ or $b \in I$.
Since $\varphi(I)=\{0\}$, we see that $\psi( a \otimes b) = 0$ whenever $ab = 0$, and thus $\psi$ vanishes on $Z_2(A)$. 
On the other hand, we have
\[
\psi(t) 
= \varphi( d_0 \otimes x ) \otimes \varphi( 1 \otimes x ) 
- \varphi( 1 \otimes x ) \otimes \varphi( d_0 \otimes x )
= d_0 \otimes 1 - 1 \otimes d_0
\neq 0.
\]
Thus, $t \notin Z_2(A)$, as desired. We conclude that $A$ is not zero-product
determined.
\end{exa}

\section{Semimultiplicative maps and transferrable elements}
\label{sec:SMD}

In order to study zero-balanced algebras and semimultiplicative maps, in this section we introduce the subalgebra $\mathcal{T}(A)$ of \emph{transferrable elements} in $A$, which measures \emph{how far} $A$ is from being zero-product balanced.
To study the case that~$A$ is nonunital, we introduce the set of \emph{multiplier transferrable elements} $\mathcal{T}_\mathrm{M}(A)$.

Given a $K$-algebra $A$, its multiplier algebra $M(A)$ is defined as the set of pairs $(L,R)$ of $K$-linear maps $L,R\colon A\to A$ such that
\[
xL(y)=R(x)y, \quad 
L(xy)=L(x)y, \andSep 
R(xy)=xR(y)
\] 
for all $x,y\in A$.
Then $M(A)$ is a $K$-algebra for the componentwise $K$-linear structure, and with the product of $(L_1,R_1),(L_2,R_2) \in M(A)$ given by $(L_1L_2,R_2R_1)$.
Every $a \in A$ defines an element $(L_a,R_a) \in M(A)$ with $L_a(x) := ax$ and $R_a(x) := xa$ for $x \in A$.
The map $\mu\colon A \to M(A)$, $a \mapsto (L_a,R_a)$, is multiplicative and $K$-linear and it maps $A$ onto an essential ideal of $M(A)$.
The kernel of $\mu$ is the ideal of two-sided annihilators in $A$.
In particular, if $A$ is faithful, then $\mu$ identifies $A$ with an essential ideal of $M(A)$, and the multiplier algebra is the largest unital $K$-algebra containing $A$ as an essential ideal;
see Section~1 of~\cite{Dau69MultiplierRings} for details.

We show that every idempotent of $M(A)$ belongs to $\mathcal{T}_{\mathrm{M}}(A)$; see \autoref{prp:IdemInSMD}.
It follows that an algebra $A$ is zero-product balanced whenever it is contained in the subalgebra of $M(A)$ generated by idempotents;
see \autoref{prp:IdemGenBalanced}.
This is an improvement over earlier results, since it can be used to show zero-product balancedness for certain algebras that contain no nontrivial idempotents.
For example, given any ring $R$, and $n\geq 2$, the matrix ring $M_n(R)$ is zero-product balanced, although it may not contain any nonzero idempotents.

\smallskip

Throughout this section, $K$ denotes a commutative, unital ring.

\begin{dfn}
\label{dfn:SMD}
Let $A$ be a $K$-algebra. 
We say that an element $b\in M(A)$ is \emph{(multiplier) transferrable} if $ab\otimes c-a\otimes bc \in Z_2(A)$ for all $a,c\in A$.
We write $\mathcal{T}_{\mathrm{M}}(A)$ for the set of multiplier transferrable elements, and we let $\mathcal{T}(A)=A\cap \mathcal{T}_{\mathrm{M}}(A)$ denote the set of \emph{transferrable} elements in $A$.
\end{dfn}

We omit the straightforward proof of the next result.

\begin{prp}
\label{prp:SMDIsSubalgebra}
Let $A$ be a $K$-algebra.
Then $\mathcal{T}_{\mathrm{M}}(A)$ is a subalgebra of $M(A)$, and thus $\mathcal{T}(A)$ is a subalgebra of $A$.
\end{prp}

\begin{prp}
\label{prp:CharBalancedSemimult}
Let $A$ be a $K$-algebra.
Then the following are equivalent:
\begin{enumerate}
\item
$A$ is zero-product balanced;
\item
$A = \mathcal{T}(A)$;
\item
$A \subseteq \mathcal{T}_{\mathrm{M}}(A)$;
\item
every zero-product preserving, bilinear map $A \times A \to V$ to a $K$-module $V$ is semimultiplicative.
\end{enumerate}
\end{prp}
\begin{proof}
The equivalence of~(1), (2) and~(3) follows from the definition of zero-product balancedness.
The equivalence of~(1) and~(4) follows from \autoref{prp:CharBalanced}.
\end{proof}

\begin{cor}
\label{prp:BalancedImplSemimult}
Let $A$ be a zero-product balanced $K$-algebra.
Then every linear, zero-product preserving map $\pi\colon A \to B$ to another $K$-algebra is semimultiplicative.
\end{cor}
\begin{proof}
Consider the bilinear map $\varphi \colon A \times A \to B$ given by $\varphi(a,b) := \pi(a)\pi(b)$.
Then~$\varphi$ preserves zero-products, and therefore is semimultiplicative by \autoref{prp:CharBalancedSemimult}.
We get
\[
\pi(ab)\pi(c)
= \varphi(ab,c)
= \varphi(a,bc)
= \pi(a)\pi(bc)
\]
for all $a,b,c \in A$.
\end{proof}

\begin{lma}
\label{prp:IdemInSMD}
Let $A$ be a $K$-algebra.
Then $\mathcal{T}_{\mathrm{M}}(A)$ contains all idempotents in $M(A)$.
\end{lma}
\begin{proof}
Let $a,c \in A$, and let $e \in M(A)$ be idempotent.
We need to verify that $ae \otimes c - a \otimes ec \in Z_2(A)$.
Note that $ae$ and $ec$ belong to $A$.
We have $ae(c-ec) = 0$ and therefore
\[
ae\otimes c - ae\otimes ec 
= ae\otimes (c-ec) \in Z_2(A).
\]
Similarly, since $(ae-a)ec=0$, we get 
\[
ae\otimes ec - ae\otimes c 
= (ae-a)\otimes ec \in Z_2(A).
\]
It follows that
\[
ae\otimes c - a\otimes ec
= \big( ae\otimes c - ae\otimes ec \big) + \big( ae\otimes ec - ae\otimes c \big) \in Z_2(A). \qedhere
\]
\end{proof}

By \cite[Theorem~4.1]{Bre12MultAlgDetZP}, a unital $K$-algebra that is generated by its idempotents (as a $K$-algebra) is zero-product determined as a $K$-algebra;
see also \cite[Theorem~2.15]{Bre21BookZeroProdDetermined}.
By \autoref{prp:ZPDImplBalanced}, such algebras are zero-product balanced.
We recover Bre\v{s}ar's result and generalize it to the nonunital setting:

\begin{prp}
\label{prp:IdemGenBalanced}
Let $A$ be a $K$-algebra which 
is contained in the subalgebra of $M(A)$ generated by the idempotents in $M(A)$.
Then $A$ is zero-product balanced.
\end{prp}
\begin{proof}
By \autoref{prp:IdemInSMD}, every idempotent in $M(A)$ belongs to $\mathcal{T}_{\mathrm{M}}(A)$.
Further, by \autoref{prp:SMDIsSubalgebra}, $\mathcal{T}_{\mathrm{M}}(A)$ is a subalgebra of $M(A)$, and therefore $A \subseteq \mathcal{T}_{\mathrm{M}}(A)$, by assumption.
By \autoref{prp:CharBalancedSemimult}, this implies that $A$ is zero-product balanced.
\end{proof}

\begin{exa}
\label{exa:IdemGen}
The following classes of rings are generated by their idempotents, and therefore are zero-product balanced;
see \cite[after Lemma~2.1]{Bre07CharHomoRingsIdempotents} for the first three, and \cite[Corollary~3.8]{Rob16LieIdeals} for the last one:
\begin{itemize}
\item
simple rings that contain a nontrivial idempotent
(this includes all simple, unital \ca{s} that contain a projection $p\neq 0,1$);
\item
unital rings $R$ containing an idempotent $e$ such that $e$ and $1-e$ are full, that is, $\linSpan_\ZZ ReR=\linSpan_\ZZ R(1-e)R=R$;
\item
matrix rings $M_n(R)$, where $R$ is any unital ring and $n\geq 2$;
\item
unital \ca{s} of real rank zero that have no one-dimensional irreducible representations.
\end{itemize}
\end{exa}

We now turn to matrix algebras.
If $A$ is a \emph{unital} $K$-algebra and $n\geq 2$, then the matrix algebra $M_n(A)$ is generated by its idempotents as a $K$-algebra (even as a ring);
see \cite[Corollary~2.4]{Bre21BookZeroProdDetermined}.
It follows from \autoref{prp:IdemGenBalanced} that $M_n(A)$ is zero-product balanced (equivalently by \autoref{prp:MultipleFactorization}, it is zero-product determined);
see also \cite[Theorem~2.1]{BreGraOrt09ZPDMatrixAlg} and \cite[Corollary~2.17]{Bre21BookZeroProdDetermined}
In the next theorem, we extend this result by showing that matrix algebras over \emph{not necessarily unital} algebras are always zero-product balanced.

\begin{thm}
\label{prp:MatrixAlgBalanced}
Let $A$ be a (not necessarily unital) $K$-algebra, and let $n\geq 2$.
Then the matrix algebra $M_n(A)$ is zero-product balanced (even as a ring).
\end{thm}
\begin{proof}
The matrix algebra $M_n(A)$ is naturally an ideal in $M_n(M(A))$.
This induces a natural, unital homomorphism $M_n(M(A))\to M(M_n(A))$ that is the identity on $M_n(A)$.
By \cite[Corollary~2.4]{Bre21BookZeroProdDetermined}, $M_n(M(A))$ is generated by its idempotents as a ring.
It follows that the subring of $M(M_n(A))$ generated by the idempotents contains $M_n(A)$.
Hence, the result follows from \autoref{prp:IdemGenBalanced}.
\end{proof}

\begin{qst}
Is every (nonunital) matrix algebra zero-product determined? 
\end{qst}

The following easy example shows that an algebra may be contained in the subalgebra generated by idempotents of its multiplier algebra even if it does not contain any nonzero idempotents itself.

\begin{exa}
Let $C_0((0,1])$ be the $\CC$-algebra of continuous functions $f \colon [0,1] \to \CC$ satisfying $f(0) = 0$.
The matrix algebra $M_2(C_0((0,1]))$ is naturally isomorphic to the algebra of continuous functions $g \colon [0,1] \to M_2(\CC)$ with $g(0)=0$.
It follows that $M_2(C_0((0,1]))$ contains no nonzero idempotents.
Nevertheless, it is zero-product balanced by \autoref{prp:MatrixAlgBalanced}.
\end{exa}

\section{Weighted epimorphisms}
\label{sec:Weighted}

This section contains one of the main applications of zero-product balancedness, namely \autoref{prp:WeightedFromBalanced}:
it allows one to deduce that certain zero-product preserving maps are automatically weighted homomorphisms.

\smallskip

Throughout this section, $K$ denotes a unital, commutative ring.
A $K$-algebra~$A$ is said to be \emph{idempotent} if $A=\linSpan_K A^2$, that is, for every $a\in A$ there exist $n\geq 1$ and $b_j,c_j\in A$, for $j=1,\ldots,n$, such that $a=\sum_{j=1}^n b_jc_j$.
A $K$-algebra~$A$ is said to be \emph{faithful} if whenever $a\in A$ satisfies $aA=\{0\}$ or $Aa=\{0\}$, then $a=0$.
A $K$-algebra is idempotent if and only if it is idempotent as a ring (that is, as a $\ZZ$-algebra).
Similarly, being faithful also only depends on the ring structure.

The class of rings that are idempotent and faithful includes every unital ring, every simple ring, and -- using Cohen's factorization theorem -- every Banach algebra with a bounded left (or right) approximate identity; in particular, every \ca{}.
Further, every semiprime ring is faithful, but there exist prime rings that are not idempotent.

A \emph{centralizer} on a $K$-algebra $A$ is a $K$-linear map $S\colon A\to A$ such that $aS(b)=S(ab)=S(a)b$ for all $a,b\in A$.

\begin{dfn}
\label{dfn:weighted}
A $K$-linear map $\pi\colon A\to B$ between $K$-algebras is said to be a \emph{weighted homomorphism} if there exist a multiplicative, linear map $\pi_0\colon A\to B$ and a centralizer $S$ on $B$ such that $\pi=S\circ\pi_0$.

A \emph{weighted epimorphism} is a surjective, weighted homomorphism.
A \emph{weighted isomorphism} is a bijective, weighted homomorphism.
\end{dfn}

Some authors require the centralizer in the definition of a weighted homomorphism to be bijective, for example \cite[Definition~3.2]{AlaBreExtVil09MapsPresZP}.
We do not follow this convention since we want to include the important class of positively weighted $\ast$-homomorphisms between \ca{s};
see \cite{GarThi22arX:WeightedHomo}.

Moreover, in \autoref{prp:WeightedEpi} below we show that under mild assumptions on the algebras, the factorization of a \emph{surjective}, weighted homomorphism as in \autoref{dfn:weighted} is unique and the centralizer is automatically bijective.
Thus, in this setting, weighted epimorphisms are precisely the maps arising as an epimorphism composed with a bijective centralizer on the target;
and weighted isomorphisms are precisely the maps arising as an isomorphism composed with a bijective centralizer on the target.
(This also shows that the terminologies `weighted epimorphism' and `weighted isomorphism' are unambiguous in this setting.)

\smallskip

Recall that a map $\pi \colon A \to B$ between $K$-algebras is said to be \emph{semimultiplicative} if $\pi(ab)\pi(c) = \pi(a)\pi(bc)$ for all $a,b,c \in A$.

\begin{prp}
\label{prp:ConsequenceSemimult}
Let $\pi \colon A \to B$ be a surjective, semimultiplicative, linear map between idempotent $K$-algebras.
Assume that $B$ is faithful.

Then the map $T\colon B\to B$, given by
\[
T\Big( \sum_{j=1}^n\pi(a_j)\pi(b_j) \Big)
= \sum_{j=1}^n\pi(a_jb_j)
\]
for $a_j,b_j\in A$, is a well-defined, bijective centralizer on $B$ and the composition $\pi_0 := T\circ\pi\colon A\to B$ is an epimorphism (a surjective, multiplicative, $K$-linear map).
In particular, $\pi$ is a weighted homomorphism with a bijective weight.
\end{prp}
\begin{proof}
We isolate the following fact for repeated use.

\textbf{Claim:} \emph{Let $a_1,\ldots,a_n,b_1,\ldots,b_n\in A$. Then }
\begin{equation*}\tag{4.1}\label{2.1}\sum_{j=1}^n\pi(a_j)\pi(b_j)=0 \ \mbox{ if and only if } \
 \sum_{j=1}^n\pi(a_jb_j)=0.
\end{equation*}
To prove the claim, let $x,y\in A$.
Then
\[
\sum_{j=1}^n\pi(a_jb_j)\pi(x)\pi(y)
= \sum_{j=1}^n\pi(a_j)\pi(b_jx)\pi(y)
= \sum_{j=1}^n\pi(a_j)\pi(b_j)\pi(xy).
\]

If $\sum_{j=1}^n\pi(a_j)\pi(b_j)=0$, then $\sum_{j=1}^n\pi(a_jb_j)\pi(A)\pi(A)=\{0\}$.
Using that $\pi$ is surjective, it follows that $\sum_{j=1}^n\pi(a_jb_j)BB=\{0\}$.
Since $B$ is faithful, we first deduce that  $\sum_{j=1}^n\pi(a_jb_j)B=\{0\}$, and then $\sum_{j=1}^n\pi(a_jb_j)=0$.

Now assume that $\sum_{j=1}^n\pi(a_jb_j)=0$.
Then $\sum_{j=1}^n\pi(a_j)\pi(b_j)\pi(A^2)=\{0\}$.
Using that $A=\linSpan_K A^2$ and that $\pi$ is surjective, it follows that $\sum_{j=1}^n\pi(a_j)\pi(b_j)B=\{0\}$, and so $\sum_{j=1}^n\pi(a_j)\pi(b_j)=0$. This proves the claim.
\vspace{0.15cm}

We first verify that $T$ as defined in the statement is a bijective centralizer.
Using the ``only if'' implication in the claim, and using that $B$ is idempotent and that $\pi$ is surjective, we see that $T$ is well-defined.
It is injective by the ``if'' implication.

We show that $T$ is surjective.
Let $v\in B$.
Using that $\pi$ is surjective and $A=\linSpan_K A^2$, we can choose $a_j,b_j\in A$ such that $v=\sum_j\pi(a_jb_j)$.
Then $v=T(\sum_j\pi(a_j)\pi(b_j))$.

We show that $T$ is a centralizer.
Let $b,c\in B$.
Choose $x_j,y_j,z,w\in A$ such that
\[
b=\pi(z)=\sum_j\pi(x_j)\pi(y_j) \andSep
c=\pi(w).
\]
Let $a\in A$.
Then
\begin{align*}
T(bc)\pi(a)
&= T(\pi(z)\pi(w))\pi(a)
= \pi(zw)\pi(a)
= \pi(z) \pi(wa) \\
&= \sum_j \pi(x_j)\pi(y_j) \pi(wa)
= \sum_j \pi(x_j) \pi(y_jw) \pi(a) \\
&= \sum_j \pi(x_jy_j)\pi(w) \pi(a)
= T\Big(\sum_j \pi(x_j)\pi(y_j)\Big) \pi(w) \pi(a) \\
&= T(b)c\pi(a).
\end{align*}
Using that $B$ is faithful and $\pi$ is surjective, we deduce that $T(bc)=T(b)c$.

Similarly, we have
\begin{align*}
\pi(a)T(cb)
&= \pi(a) \pi(wz)
= \pi(aw) \pi(z)
= \sum_j \pi(aw) \pi(x_j) \pi(y_j) \\
&= \sum_j \pi(a) \pi(w) \pi(x_jy_j)
= \pi(a) cT(b),
\end{align*}
which implies $T(cb)=cT(b)$.
This shows that $T$ is a centralizer.

To show that $\pi_0:=T\circ\pi$ is multiplicative, let $a,b\in A$.
Using that $T$ is a centralizer at the first two steps, we get
\[
T(\pi(a)) T(\pi(b))
= T\big( \pi(a) T(\pi(b)) \big)
= T\Big( T(\pi(a)\pi(b)) \Big)
= T( \pi(ab) ).
\]

Using that $T$ is bijective, it is clear that $\pi_0$ is surjective.
\end{proof}

\begin{thm}
\label{prp:WeightedEpi}
Let $\pi \colon A \to B$ be a surjective, linear map between idempotent $K$-algebras.
Assume that $B$ is faithful.

Then the following are equivalent:
\begin{enumerate}
\item
$\pi$ is semimultiplicative: 
$\pi(ab)\pi(c)=\pi(a)\pi(bc)$ for all $a,b,c\in A$;
\item
$\pi$ is a weighted homomorphism: 
there exist a homomorphism $\pi_0\colon A\to B$ and a centralizer $S$ on $B$ such that $\pi(a)=S(\pi_0(a))$ for every $a\in A$.
\end{enumerate}

Moreover, if the above hold, then the homomorphism $\pi_0$ and the centralizer $S$ as in~(2) are unique, and $\pi_0$ is surjective, and $S$ is bijective and satisfies
\begin{equation}\tag{4.2}
\label{eq:WeightedEpi}
S\big( \pi( ab ) \big)
= \pi(a)\pi(b)
\end{equation}
for all $a,b\in A$.
If these conditions hold, then $B$ is isomorphic to a quotient of $A$.
Specifically, $I=\{a\in A \colon \pi(a)=0\}$ is an ideal and $A/I\cong B$.

Further, $\pi$ is bijective (that is, a weighted isomorphism) if and only if $\pi_0$ is bijective (that is, an isomorphism).
\end{thm}
\begin{proof}
To verify that~(2) implies~(1), assume that $\pi=S\circ\pi_0$ for a multiplicative, $K$-linear map $\pi_0\colon A\to B$ and a centralizer $S$ on $B$.
Given $a,b,c\in A$, we have
\begin{align*}
\pi(ab)\pi(c)
&= S(\pi_0(ab))S(\pi_0(c))
= S(\pi_0(a)\pi_0(b))S(\pi_0(c)) \\
&= S(\pi_0(a))\pi_0(b)S(\pi_0(c)) 
= S(\pi_0(a))S(\pi_0(b)\pi_0(c))
= \pi(a)\pi(bc).
\end{align*}

Conversely, to verify that~(1) implies~(2), assume that $\pi$ is semimultiplicative.
It follows from \autoref{prp:ConsequenceSemimult} that there exist a surjective homomorphism $\pi_0\colon A\to B$ and a bijective centralizer $S$ on $B$ satisfying \eqref{eq:WeightedEpi} such that $\pi(a)=S(\pi_0(a))$ for every $a\in A$.
This shows that~(2) holds.

To verify that $\pi_0$ and $S$ are unique, let $\widetilde{\pi_0}\colon A\to B$ be a homomorphism and let~$\widetilde{S}$ be a centralizer on $B$ such that $\pi(a)=\widetilde{S}(\widetilde{\pi_0}(a))$ for every $a\in S$.
A priori, $\widetilde{\pi_0}$ is not necessarily surjective and $\widetilde{S}$ is not necessarily bijective.
We first show that~$\widetilde{S}$ satisfies \eqref{eq:WeightedEpi}.
Given $a,b\in A$, we have
\begin{align*}
\widetilde{S}\big( \pi(ab) \big)
&= \widetilde{S}\big( \widetilde{S}(\widetilde{\pi_0}(a)\widetilde{\pi_0}(b)) \big)
= \widetilde{S}\big( \widetilde{\pi_0}(a) \widetilde{S}(\widetilde{\pi_0}(b)) \big)
= \widetilde{S}\big( \widetilde{\pi_0}(a) \pi(b) \big) \\
&= \widetilde{S}\big( \widetilde{\pi_0}(a) \big) \pi(b) 
= \pi(a) \pi(b).
\end{align*}

Using that $\pi$ is surjective and $A$ is idempotent, every element in $B$ is of the form $\sum_j \pi(a_jb_j)$ for some $a_j,b_j\in A$, and it follows that $S=\widetilde{S}$.
In particular, $\widetilde{S}$ is bijective.
It then follows that $\widetilde{\pi_0}=\pi_0$.

Using that $S$ is bijective, it follows that $I$ (the kernel of $\pi$) agrees with the kernel of $\pi_0$, which clearly is an ideal of $A$.
Further, $\pi_0$ induces an isomorphism between~$A/I$ and $B$.
\end{proof}

The next result presents a major application of being zero-product balanced.

\begin{thm}
\label{prp:WeightedFromBalanced}
Let $\pi \colon A \to B$ be a surjective, linear map between idempotent $K$-algebras.
Assume that $A$ is zero-product balanced and that $B$ is faithful.

Then the following are equivalent:
\begin{enumerate}
\item
$\pi$ preserves zero-products;
\item
$\pi$ is semimultiplicative; 
\item
$\pi$ is a weighted homomorphism.
\end{enumerate}

If these conditions hold, then $B$ is isomorphic to a quotient of $A$.
Specifically, $I=\{a\in A : \pi(a)=0\}$ is an ideal and $A/I\cong B$.
\end{thm}
\begin{proof}
By \autoref{prp:BalancedImplSemimult}, (1) implies~(2).
The equivalence between~(2) and~(3) follows from \autoref{prp:WeightedEpi}.
To show that~(2) implies~(1), let $a,b\in A$ satisfy $ab=0$.
Given $x,y\in A$, it follows that
\[
\pi(a)\pi(b)\pi(xy)
= \pi(a)\pi(bx)\pi(y)
= \pi(ab)\pi(x)\pi(y)
= 0.
\]
Using that $A=\linSpan_K A^2$ and that $\pi$ is surjective, it follows that $\pi(a)\pi(b)B=\{0\}$.
Since $B$ is faithful, we get $\pi(a)\pi(b)=0$, as desired.

Assuming that the conditions~(1)-(3) hold, the remaining statements follow from \autoref{prp:WeightedEpi}.
\end{proof}

\begin{cor}
\label{prp:IsoDetByZPStructure}
Let $A$ be a zero-product balanced, idempotent $K$-algebra, and let~$B$ be a faithful, idempotent $K$-algebra.
Then $A$ and $B$ are isomorphic as $K$-algebras if and only if there is a bijective, zero-product preserving, linear map $A \to B$. 
\end{cor}

\begin{exa}
\label{exa:MatrixAlgIso}
Let $A$ and $B$ be idempotent $K$-algebras, let $n \geq 2$, and assume that $B$ is faithful.
Then $M_n(A)$ and~$B$ are isomorphic as $K$-algebras if and only if there exists a bijective, zero-product preserving, linear map $M_n(A) \to B$.

Indeed, using idempotency of $A$, one easily shows that $M_n(A)$ is idempotent.
Further, $M_n(A)$ is zero-product balanced by \autoref{prp:MatrixAlgBalanced}.
Now the result follows from \autoref{prp:IsoDetByZPStructure}.
\end{exa}

\begin{rmk}
\label{rmk:UnitlWeightedEpi}
Let $\pi\colon A\to B$ be a surjective, semimultiplicative, linear map between $K$-algebras.
Assume that $A$ is \emph{unital}, and that~$B$ is idempotent and faithful.
Then \autoref{prp:WeightedFromBalanced} applies and we obtain a surjective homomorphism $\pi_0\colon A\to B$ and a bijective centralizer $S\colon B\to B$ such that $\pi=S\circ\pi_0$.
Further, $B$ is isomorphic to a quotient of $A$ and thus $B$ is unital.

Let us see that $\pi(1)$ is a central invertible element in $B$, and that $S$ is given by $S(b)=\pi(1)b$ for $b\in B$.
Indeed, by \eqref{eq:WeightedEpi}, $S$ satisfies
\[
S(\pi(a)) = S(\pi(1a)) = \pi(1)\pi(a) \andSep
S(\pi(a)) = S(\pi(a1)) = \pi(a)\pi(1),
\]
for every $a\in A$.
Using that~$\pi$ is surjective, we deduce that $\pi(1)$ is central in $B$, and that $S(b)=\pi(1)b=b\pi(1)$ for every $b\in B$.
Since $S$ is bijective, it follows that~$\pi(1)$ is invertible in $B$. 
\end{rmk}

\section{Commutators and factorizable square-zero elements}
\label{sec:FN2}

This section contains another important application of zero-product balancedness:
A description of the commutator subspace in terms of square-zero elements.
We refer to \cite[Section~9.1]{Bre21BookZeroProdDetermined} for an introduction to the historical context and for analogous results for commutators in zero-product determined algebras.

To obtain our result, we introduce the class of \emph{orthogonally factorizable} square-zero elements;
see \autoref{dfn:FN2}.
We show that in idempotent, zero-product balanced algebras, the subspace generated by these elements agrees with the commutator subspace;
see \autoref{prp:BalancedCommutatorFN2}.
In particular, in idempotent, zero-product balanced rings, every commutator is a sum of (orthogonally factorizable) square-zero elements.

\smallskip

Throughout this section, $K$ denotes a unital, commutative ring.
Our results apply to $K$-algebras, which for $K = \ZZ$ includes the case of rings.

\begin{dfn}
\label{dfn:FN2}
Let $A$ be a $K$-algebra.
We let $N_2(A) := \{ x\in A : x^2=0 \}$ denote the set of square-zero elements.
We further set
\begin{align*}
FN_2(A) &:= \big\{ x\in A : \text{ there exist } y,z \in A \text{ with } x=yz \mbox{ and } zy=0 \big\}.
\end{align*}
We call the members of $FN_2(A)$ \emph{orthogonally factorizable} square-zero elements.

If the context is clear, we simply write $N_2$ for $N_2(A)$ and $FN_2$ for $FN_2(A)$.
\end{dfn}

Given a $K$-algebra $A$, we denote the additive commutator of elements $a,b \in A$ by $[a,b]:=ab-ba$.
We set $[A,A] := \{ [a,b] : a,b \in A \}$, and we let $\linSpan_K [A,A]$ denote the linear subspace of $A$ generated by the commutators.
Note that many authors use $[A,A]$ to denote $\linSpan_K [A,A]$.

\begin{lma}
\label{prp:FN2-RR}
Let $A$ be a $K$-algebra.
Then $FN_2(A) \subseteq [A,A]\cap N_2(A)$.
\end{lma}
\begin{proof}
Let $x\in FN_2(A)$.
Choose $y,z \in A$ with $x=yz$ and $zy=0$.
Then $x = yz-zy = [y,z] \in [A,A]$. 
Also, $x^2 = yzyz = 0$, so $x \in N_2(A)$ as well.
\end{proof}

Recall that a $K$-algebra $A$ is said to be idempotent if $A = \linSpan_K A^2$.

\begin{thm}
\label{prp:BalancedCommutatorFN2}
Let $A$ be a zero-product balanced, idempotent $K$-algebra.
Then
\[
\linSpan_K [A,A]
\ = \ \linSpan_K FN_2(A).
\]
\end{thm}
\begin{proof}
By \autoref{prp:FN2-RR}, the inclusion `$\supseteq$' holds in general.
The proof of the reverse inclusion is inspired by the proof of \cite[Theorem 9.1]{Bre21BookZeroProdDetermined}.
Set $F := \linSpan_K FN_2$, and consider the biadditive map $\varphi \colon A \times A\to A/F$
given by $\varphi(x,y) := yx+F$ for $x,y\in A$.
Then $\varphi$ preserves zero-products, since if $x,y \in A$ satisfy $xy = 0$, then $yx \in FN_2(A)\subseteq F$, that is, $\varphi(x,y)=0+F$.
Since $A$ is zero-product balanced, for all $x,y,z \in A$ we obtain $\varphi(xy,z) = \varphi(x,yz)$ and hence
\[\tag{5.1}\label{5.1}
zxy-yzx \in F.
\]

Now, given $a,b \in A$, we need to verify $[a,b] \in F$.
Since $A$ is idempotent, we can choose $z_j,x_j \in A$ such that $a = \sum_{j=1}^n z_jx_j$.
Using \eqref{5.1}, we get
\[
[a,b]
= \sum_{j=1}^n (z_jx_jb - bz_jx_j) \in F 
= \linSpan_K FN_2. \qedhere
\]
\end{proof}

\begin{exa}
Let $R$ is a simple ring that contains a nontrivial idempotent.
Then \cite[Theorem~6]{CheLeePuc10CommNilSimpleRings} shows that $[R,R]\subseteq\linSpan_\ZZ N_2$.
We recover and strengthen this result as follows.
As noted in \autoref{exa:IdemGen}, $R$ is generated by its idempotent elements and is therefore zero-product balanced by \autoref{prp:IdemGenBalanced}.
Since every simple ring is idempotent, \autoref{prp:BalancedCommutatorFN2} applies and shows that
\[
\linSpan_\ZZ [R,R]
\ = \ \linSpan_\ZZ FN_2
\ \subseteq \ \linSpan_\ZZ N_2.
\]

In ,  we study the case of more general simple rings.

In \cite{GarThi23pre:SimpleRgsSpecialSqZero}, we study the case of more general \emph{simple} rings, while in \cite{GarThi23pre:ZeroProdRingsCAlgs} arbitrary rings are
studied. This, in particular, applies to C*-algebras, which are one of the main classes
considered in \cite{GarThi23pre:ZeroProdRingsCAlgs}. It would also be interesting to explore aspects of this nature in the
context of $L^p$-operator algebras \cite{Gar21ModernLp}, for example for the well-behaved families arising
from groups \cite{GarThi15GpAlgLp, GarThi22IsoConv} 
or groupoids \cite{GarLup17ReprGrpdLp, ChoGarThi19arX:LpRigidity}.
\end{exa}

\begin{exa}
Let $R$ be a left Rickart ring, that is, assume that for every $x\in R$ there is an idempotent $e\in R$ such that the left annihilator $\{a\in R: ax=0\}$ is equal to the left ideal $Re$.
(For example, this is the case if $R$ is a von Neumann regular ring, or an AW*-algebra, or a von Neumann algebra.)
Let us see that $FN_2=N_2$.

The inclusions $FN_2\subseteq N_2$ is true in general by \autoref{prp:FN2-RR}.
To show the converse inclusion, let $x\in N_2$.
Find an idempotent $e\in R$ such that
\[
Re = \{ a\in R : ax=0 \}.
\]
Since $x^2=0$, we have $x\in Re$.
Choose $r\in R$ such that $x=re$.
Then $xe=ree=re=x$.
Set $f:=1-e$.
Since $e=ee\in Re$, we have $ex=0$, and thus $fx=x$.
Thus, $x=xe$ and $ex=0$, which shows $x\in FN_2$.

Hence, if $R$ is a left Rickart ring that is zero-product balanced, then
\[
\linSpan_\ZZ [R,R] 
\ = \ \linSpan_\ZZ FN_2
\ = \ \linSpan_\ZZ N_2.
\]
\end{exa}

\begin{exa}
Let $A$ be an idempotent $K$-algebra, and let $n\geq 2$.
Given a matrix $a \in M_n(A)$, the following are equivalent:
\begin{enumerate}
\item
$a \in \linSpan_K [M_n(A),M_n(A)]$;
\item
$a \in \linSpan_K FN_2(M_n(A))$;
\item
$\trace(a) \in \linSpan_K [R,R]$.
\end{enumerate}

Indeed, as in \autoref{exa:MatrixAlgIso}, we see that $M_n(A)$ is idempotent and zero-product balanced.
Therefore, the equivalence between~(1) and~(2) follows from \autoref{prp:BalancedCommutatorFN2}.
By \cite[Corollary~17]{Mes06CommutatorRings}, (1) is equivalent to~(3) for arbitrary rings.
\end{exa}

\begin{rmk}
For $n\geq 2$, it is natural to try determine the minimal numbers $N(n)$ and $N_F(n)$ such that every commutator in the matrix ring $M_n(R)$ (with $n\geq 2$) over a unital ring $R$ is a sum of at most $N(n)$ square-zero elements (at most $N_F(n)$ orthogonally factorizable square-zero elements).
It was shown in \cite[Theorem~4.4]{AlaExtVilBreSpe16CommutatorsSquareZero} that every commutator in the matrix algebra over a unital algebra is a sum of at most 22 square-zero elements.
The result also holds for unital rings, and an inspection of the proof shows that the constructed square-zero elements are orthogonally factorizable, and thus $N_F(n)\leq 22$ for all $n\geq 2$.

On the other hand, it was shown in Corollary~3.4 and Theorem~3.6 of \cite{WanWu91SumsSquareZero} that the diagonal matrix in $M_5(\CC)$ with diagonal entries $4,-1,-1,-1,-1$ is not the sum of three square-zero matrices, while it clearly is a commutator:
\[
\begin{pmatrix}
4 & & & & 0 \\
& -1 \\
& & -1 \\
& & & -1 \\
0 & & & & -1
\end{pmatrix}
= \left[
\begin{pmatrix}
0 & 4 & & & 0 \\
& & 3 \\
& & & 2 \\
& & & & 1 \\
0 & & & & 0
\end{pmatrix},
\begin{pmatrix}
0 & & & & 0 \\
1 \\
& 1 \\
& & 1 \\
& & & 1 & 0
\end{pmatrix}
\right]
\]
We thus have $4\leq N(5)\leq N_F(5)\leq 22$.
\end{rmk}

\begin{pbm}
Determine $N(n)$ and $N_F(n)$ for $n\geq 2$.
\end{pbm}

\section{Zero-product balanced algebras without zero divisors}
\label{sec:DivisionAlg}

As always, we will assume that $K$ is a unital, commutative ring.
We start with an elementary observation:

\begin{rmk}
\label{prp:DomainBasic}
Let $A$ be a zero-product balanced $K$-algebra that contains no zero-divisors.
Then
\[
Z_2(A) := \linSpan_K \big\{ u \otimes v \in A \otimes_K A : uv = 0 \big\} 
= \{ 0 \},
\]
and thus $ab \otimes c = a \otimes bc$ in $A \otimes_K A$ for every $a,b,c \in A$.
Further, every $K$-bilinear map $\varphi \colon A \times A \to V$ to a $K$-module automatically preserves zero-products, and thus
\[
\varphi(ab,c) = \varphi(a,bc)
\]
for every $a,b,c \in A$, by \autoref{prp:CharBalanced}.
\end{rmk}

The last statement in the next result is a generalization of \cite[Corollary~1.8]{Bre21BookZeroProdDetermined} to the nonunital setting.

\begin{prp}
\label{prp:DomainOverField}
Let $A$ be a zero-product balanced $K$-algebra that contains no zero-divisors.
Then $A$ is commutative. Moreover, if $K$ is a field, then either $A=\{0\}$ or $A\cong K$.
\end{prp}
\begin{proof}
We may assume that $A\neq\{0\}$.
To show that $A$ is commutative, let $a,b \in A$.
Choose nonzero elements $x,y \in A$.
Applying \autoref{prp:DomainBasic} for the biadditive map $\varphi \colon R\times R\to R$, $(r,s)\mapsto rbs$, we obtain 
\[
xaby
= \varphi(xa,y)
= \varphi(x,ay)
= xbay.
\]
Hence, $x(ab-ba)y=0$, and since $A$ has no zero divisors, we deduce that $ab=ba$, as desired.

Assume now that $K$ is a field and that $A\neq \{0\}$. 
We will show that every two nonzero elements of $A$ are scalar multiples of each other, which readily implies the conclusion.
Let $a,b \in A$ be nonzero, and choose any other nonzero element $c \in A$.
Then $bc \neq 0$ and $ca \neq 0$, which allows us to choose a $K$-linear functional $\varphi \colon A \to K$ such that $\varphi(bc) \neq 0$ and $\varphi(ca) \neq 0$.

Applying \autoref{prp:DomainBasic} to the $K$-bilinear map $A \times A \to A$ given by $(x,y) \mapsto x\varphi(y)$, at the first step, and for the map $A \times A \to A$ given by $(x,y) \mapsto \varphi(x)y$, at the third step, we get
\[
a\varphi(bc)
= ab\varphi(c)
= \varphi(c)ab
= \varphi(ca)b,
\]
Hence, $a=\varphi(bc)^{-1}\varphi(ca)b$, as desired.
\end{proof}

\begin{rmk}
Let $A$ be a $K$-algebra such that
\[
ab \otimes c = a \otimes bc
\]
in $A \otimes_K A$ for every $a,b,c \in A$.
Applying \autoref{prp:DomainBasic} for the biadditive map $\varphi \colon R\times R\to R$, $(r,s)\mapsto sr$, we obtain 
\[
abx
= \varphi(bx,a)
= \varphi(b,xa)
= xab
\]
for all $a,b,c \in A$.
Hence, every product of two elements belongs to the center $Z(A)$, and it follows that $Z(A)$ is an ideal in $A$ such that the quotient $A/Z(A)$ has zero multiplication.
Further, the product of a commutator with the product of any two elements is zero, and it follows that the commutator ideal in $A$ has zero multiplication.

Does it follow that $A$ is commutative?
The answer is `Yes' if $A$ is faithful (for example, if it is unital or has no zero divisors), or idempotent, or has no nilpotent ideal (for example, if it is semiprime), or if $K$ is a field (\autoref{prp:DomainOverField}).
\end{rmk}

\begin{qst}
Does there exist a noncommutative ring $R$ such that $ab \otimes c = a \otimes bc$ in $R \otimes_\ZZ R$ for all $a,b,c \in R$?
\end{qst}

Following \cite{BouKan72Core}, we say that a unital, commutative ring $R$ is \emph{solid} if $x\otimes 1=1\otimes x$ in $R\otimes_\ZZ R$ for every $x\in R$.
It follows immediately that every solid ring is zero-product balanced.
Typical examples of solid rings are the finite cyclic rings $\ZZ/n\ZZ$ for $n\geq 2$, and the unital subrings of the field $\QQ$ of rationals.
The classification of solid rings from \cite[Proposition~3.5]{BouKan72Core} shows that every solid ring can be constructed from these typical examples.

\begin{prp}
\label{prp:CharDomainZPBRing}
Let $R$ be a unital ring without zero divisors.
Then $R$ is zero-product balanced as a ring if and only if it is a (commutative) solid ring.
\end{prp}
\begin{proof}
Using that $Z_2(R)=\{0\}$, we see that $R$ is zero-product balanced as a ring if and only if $ab\otimes c=a\otimes bc$ for all $a,b,c\in R$.
Using that $R$ is unital, this is equivalent to $x\otimes 1=1\otimes x$ for all $x\in R$.
By definition, this holds if $R$ is solid, which proves the backward direction.
Conversely, assume that $R$ is zero-product balanced.
Then $R$ is commutative by \autoref{prp:DomainOverField}, and thus a solid ring.
\end{proof}

\begin{exa}
\label{exa:CnotZPB}
While the field $\QQ$ is zero-product balanced as a ring, the fields $\mathbb{R}$ and $\mathbb{C}$ are not.
Indeed, $\RR$ and $\CC$ are not solid, since every torsion-free, solid ring is a subring of $\QQ$ by \cite[Lemma~3.7]{BouKan72Core}. On the other hand,
the ring $2\ZZ$ of even integers is a nonunital, zero-product balanced ring without zero divisors.
\end{exa}

\section{Commutative, zero-product balanced algebras}
\label{sec:Commutative}

In this section, we specialize to the case that $K$ is a field.
We prove the main result of the paper:
A commutative $K$-algebra over a field $K$ is reduced and zero-product balanced if and only if it is generated by idempotents;
see \autoref{prp:SemiprimeCommutative}.
This is in the same spirit as the result of Bre\v{s}ar, \cite[Theorem~3.7]{Bre16FinDimZeroProdDet}, that a \emph{unital}, finite-dimensionl $K$-algebra is zero-product balanced (equivalently, zero-product determined) if and only if it is generated by idempotents.
We deduce that every commutative, zero-product balanced algebra is spanned by its idempotent and nilpotent elements;
see \autoref{prp:CommutativeSpanned}.

\smallskip

Let $R$ be a ring (not necessarily unital or commtuative).
A right ideal $I \subseteq R$ is \emph{regular} if there exists $a \in R$ such that $ax - x \in I$ for all $x \in R$. 
The \emph{Jacobson radical} of $R$, denoted $\rad(R)$, is the intersection of all maximal regular right ideals of $R$.
If $R$ is unital, then every right ideal is regular.
If $R$ is commutative, then regular right ideals are precisely the (two-sided) ideals $I \subseteq R$ such that $R/I$ is unital, and it follows that $\rad(R)$ is the intersection of all ideals $I \subseteq R$ such that $R/I$ is a field.
The ring $R$ is \emph{semiprimitive} (or \emph{Jacobson semisimple}) if $\rad(R)=\{0\}$.
For details we refer to \cite[Section~I.2]{Her94NCRings}.

The \emph{prime radical} (also called \emph{lower nilradical}) of $A$, denoted $\nil(R)$, is the intersection of all prime ideals of $R$.
If $R$ is commtuative, then $\nil(R)$ is the set of nilpotent elements in $R$.
The ring $R$ is \emph{semiprime} if $\nil(R)=\{0\}$.
Note that $\nil(R) \subseteq \rad(R)$.

A ring is said to be \emph{reduced} if it contains no nonzero nilpotent elements.
A commutative ring is reduced if and only if it is semiprime.

\begin{lma}
\label{prp:Radicals}
Let $A$ be a commutative, zero-product balanced $K$-algebra over a field~$K$.
Then $A/I \cong K$ for every proper prime ideal $I \subseteq A$.

Thus, every proper prime ideal of $A$ is maximal regular, and $\nil(A) = \rad(A)$.
\end{lma}
\begin{proof}
Let $I\subseteq A$ be a proper prime ideal in $A$.
By \autoref{prp:BalancedQuotient}, the quotient $A/I$ is a zero-product balanced $K$-algebra.
Since it is also prime and commutative, and therefore has no zero divisors, it follows that $A/I \cong K$ by \autoref{prp:DomainOverField}.
Thus, $A/I$ is unital and simple, and therefore $I$ is a maximal regular ideal.
\end{proof}

The next result generalizes \cite[Example~2.24]{Bre21BookZeroProdDetermined}.

\begin{prp}
\label{prp:InfiniteDistinct}
Let $K$ be a field, and let $A$ be a (not necessarily unital) subalgebra of $\prod_\NN K$.
If $A$ contains an element whose components are all pairwise distinct, then~$A$ is not zero-product balanced as a $K$-algebra.
\end{prp}
\begin{proof}
The proof is based on that of \cite[Example~2.24]{Bre21BookZeroProdDetermined},.
We view elements of~$A$ as sequences $(a_n)_n$ with $a_n \in K$.
By assumption, there exists $y = (y_n)_n \in A$ such that the elements $y_n$ are pairwise distinct.
To reach a contradiction, we assume that $A$ is zero-product balanced.
By definition, we deduce that $y^2 \otimes y - y \otimes y^2$ belongs to $Z_2(A)$, that is, there exist $u_k, v_k \in A$ such that 
\[
y^2 \otimes y - y \otimes y^2 = \sum_{k=1}^n u_k \otimes v_k,
\]
and such that $u_kv_k = 0$ for each $k = 1, \ldots, n$.
As in the proof of \cite[Example~2.24]{Bre21BookZeroProdDetermined}, there is an infinite subset $N \subseteq \NN$ such that, with 
\[
I := \big\{ a \in A : a_n = 0 \text{ for $n \in N$} \big\}
\]
we have $u_k \in I$ or $v_k \in I$ for each $k = 1, \ldots, n$.
Grouping the tensors $u_k \otimes v_k$ according to whether $u_k \in I$ or $v_k \in I$, we rewrite
\begin{equation}\tag{7.1}
\label{eq:InfiniteDistinct}
y^2 \otimes y - y \otimes y^2 + \sum_{i=1}^l s_i \otimes z_i = \sum_{j=1}^m w_j \otimes t_j
\end{equation}
for $s_i, t_j \in I$ and $z_i, w_j \in A$.
We may assume that the elements $s_1, \ldots, s_l$ are linearly independent.
To show that the elements $y^2, y, s_1, \ldots, s_l$ are linearly independent, assume that $\lambda,\mu,\kappa_1,\ldots,\kappa_l \in K$ satisfy
\[
0 = \lambda y^2 + \mu y + \sum_{i=1}^l \kappa_i s_i.
\]
Considering indices in $s,t \in N$ where $y_s$ and $y_t$ are nonzero and distinct, we obtain $\lambda = \nu = 0$, and then also $\kappa_1 = \ldots = \kappa_l = 0$.

Let $\varphi \colon A \to K$ be a $K$-linear map vanishing on $y, s_1, \ldots, s_l$, and such that $\varphi(y^2) = 1$, which exists because $\{y^2, y,s_1,\ldots,s_l\}$
is linearly independent.
Applying the $K$-bilinear map $A \times A \to A$ given by $(x,y) \mapsto \varphi(x)y$, to \eqref{eq:InfiniteDistinct}, we obtain that
\[
y = \sum_{j=1}^m \varphi(w_j)t_j \in I,
\]
which is the desired contradiction.
\end{proof}

Let us briefly recall the Stone duality between Boolean rings and Boolean spaces.
A \emph{Boolean ring} is a ring such that every element is idempotent.
A \emph{Boolean space} is a zero-dimensional, locally compact, Hausdorff space.
Given a Boolean ring~$R$, the set of nonzero ring homomorphisms $R \to \{0,1\}$ can be equipped with a natural topology turning it into a Boolean space $X_R$.

Conversely, given a Boolean space $X$, the space $R_X:=C_c(X,\{0,1\})$ of continuous functions $X \to \{0,1\}$ that have compact support is naturally a Boolean ring.
Moreover, there are natural isomorphisms $R \cong R_{X_R}$ and $X \cong X_{R_X}$, and this even induces an equivalence of categories for suitable choices of morphisms.

We note that this duality restricts to a duality between \emph{unital} Boolean rings and \emph{compact} Boolean spaces.

Let $R$ be a commutative ring.
We turn the set $\Idem(R)$ of idempotent elements in $R$ into a Boolean ring by using the same multiplication as in $R$, but by defining a new addition $\bar{+}$ by $e\bar{+}f := e+f-2ef$.

We write $K_d$ for the field $K$ equipped with the discrete topology.
Given a locally compact, Hausdorff space $X$, we use $C_c(X,K_d)$ to denote the algebra of continuous functions $f \colon X \to K_d$ that have compact support.

Let $A$ be a commutative $K$-algebra.
We say that two idempotents $e,f \in A$ are orthogonal if $ef = 0$.
Using that the product of two idempotents in $A$ is again an idempotent, it follows that $A$ is generated by idempotents as a $K$-algebra if and only if it is the $K$-linear span of its idempotents.
We simply say that $A$ is generated by idempotents.
If this is the case, then every element $a \in A$ can be written as $a = \sum_{j=1}^n \lambda_j e_j$ for nonzero, pairwise different coefficients $\lambda_j \in K$ and pairwise orthogonal idempotents $e_1, \ldots, e_n$ in $A$.
Moreover, this presentation is unique up to permutation of indices.

\begin{thm}
\label{prp:SemiprimeCommutative}
Let $A$ be a commutative $K$-algebra over a field $K$.
Then the following are equivalent:
\begin{enumerate}
\item
$A$ is reduced and zero-product balanced as a $K$-algebra;
\item
$A$ is generated by idempotents as a $K$-algebra;
\item
$A \cong C_c(X,K_d)$ for a Boolean space~$X$.
\end{enumerate}   
Moreover, if these statements are satisfied, then the space $X$ in~(3) is unique (up to homeomorphism) and $X \cong X_{\Idem(R)}$, the Boolean space associated to $\Idem(R)$.
\end{thm}
\begin{proof}
To show that~(2) implies~(1), assume that $A$ is generated by idempotents.
Then $A$ is zero-product balanced by \autoref{prp:IdemGenBalanced}.
Next, let $a \in A$ satisfy $a^2 = 0$.
Choose nonzero, pairwise different $\lambda_j \in K$ and pairwise orthogonal idempotents $e_1, \ldots, e_n \in A$ such that $a = \sum_{j=1}^n \lambda_j e_j$.
One readily gets $0=a^2 = \sum_{j=1}^n \lambda_j^2 e_j$, and thus 
$\lambda_1=\ldots=\lambda_n=0$, which implies $a=0$.
It follows that $\nil(A)=\{0\}$, and so $A$ is reduced.

To show that~(1) implies~(2), assume that $A$ is reduced and zero-product balanced.
Let $X$ denote the set of prime ideals in $A$.
By \autoref{prp:Radicals}, for each $x\in X$ we obtain a surjective ring homomorphism $\pi_x \colon A \to K$ with kernel $x$.
Given $a \in A$, we write $a(x) := \pi_x(a) \in K$.
Since $A$ is reduced and commutative, it is semiprime, and so the intersection of all prime ideals is zero.
It follows that the map $\pi \colon A \to \prod_{x \in X} K$, given by 
$\pi(a) := (a(x))_{x \in X}$, is an injective ring homomorphism.

Given $a \in A$, we set
\[
\sigma(a) := \big\{ a(x) : x\in X \big\}.
\]

We claim that $\sigma(a)$ is finite.
To obtain a contradiction, assume that there exists a sequence $(x_n)_n$ in $X$ such that the elements $a(x_n)$ are pairwise distinct.
Then the composition of $\pi$ with the canonical map $\prod_X K\to \prod_\NN K$ 
is a ring homomorphism whose image is a subring of $\prod_\NN K$ that contains an element whose components are pairwise distinct.
By \autoref{prp:InfiniteDistinct}, $\pi(A)$ is not zero-product balanced, which contradicts \autoref{prp:BalancedQuotient}. This proves the claim.

To show that $A$ is generated by idempotents, let $a \in A$ with $a \neq 0$.
Let $\lambda_1, \ldots, \lambda_n$ be an enumeration of the nonzero elements in $\sigma(a)$.
For each $j = 1, \ldots, n$ set $F_j := \{ x\in X : a(x) = \lambda_j \}$, and let $f_j \in \prod_{x\in X} K$ be the indicator function of~$F_j$, that is, $f_j = ((f_j)_x)_{x\in X}$ with $(f_j)_x = 1$ if $x \in F_j$ and $(f_j)_x = 0$ otherwise.
Then $\pi(a) = \sum_{j=1}^n \lambda_j f_j$, and it suffices to show that the idempotents $f_1, \ldots, f_n$ belong to~$\pi(A)$.

Fix $j \in \{1, \ldots, n\}$.
Then $\lambda_j \prod_{k \neq j} (\lambda_j - \lambda_k) \in K$ is nonzero, and we can define
\[
e_j := \left(  \lambda_j \prod_{k \neq j} (\lambda_j - \lambda_k)  \right)^{-1} a \prod_{k \neq j} (a - \lambda_k) \in A,
\]
where $a \prod_{k \neq j} (a - \lambda_k)$ is understood as $\sum_{S\subseteq \{1, \ldots, n\} \setminus \{j\}} (\prod_{k\in S} \lambda_k) a^{n-|S|}$. Note that
$e_j$ is an idempotent.
Evaluating at each $x \in X$, we see that $\pi(e_j) = f_j$.
We have shown that  $a = \sum_{j=1}^n \lambda_j e_j$ for idempotents $e_1, \ldots, e_n \in A$.

To show that~(3) implies~(2), we just notice that $C_c(X,K_d)$ is generated by idempotents.
To show the converse, assume that $A$ is generated by idempotents.
Set $X := X_{\Idem(A)}$.
Given an idempotent $e \in A$, let $S(e)$ denote the subset of $X$ consisting of those characters on $\Idem(A)$ that map $e$ to $1$.
Then $S(e)$ is a compact, open subset of $X$, and the characteristic function $\chi_{S(e)}$ belongs to $C_c(X,\{0,1\})$.
The isomorphism $\Idem(A) \cong C_c(X,\{0,1\})$ of Boolean rings is given by $e \mapsto \chi_{S(e)}$.

We define $\pi \colon A \to C_c(X,K_d)$ by
\[
\sum_{j=1}^n \lambda_j e_j \mapsto \sum_{j=1}^n \lambda_j \chi_{S(e_j)},
\]
for $\lambda_j \in K$ and pairwise orthogonal idempotents $e_j \in A$.
One checks that this map is an isomorphism of $K$-algebras.
\end{proof}

Observe that the one-dimensional, commutative algebra $K$, equipped with zero multiplication, is zero-product balanced (even zero-product determined by \autoref{exa:ZeroProduct}), yet not generated by idempotents.
This shows that \autoref{prp:SemiprimeCommutative} does not generalize to arbitrary commutative algbras, and that \cite{Bre16FinDimZeroProdDet} does not generalize to arbitrary (nonunital) finite-dimensional algebras.

A ring $R$ is said to be \emph{regular} if for every $x \in R$ there is $y \in R$ with $x = xyx$.

\begin{cor}
\label{prp:QuotientByNil}
Let $A$ be a commutative, zero-product balanced $K$-algebra over a field $K$.
Then $A/\nil(A)$ is generated by idempotents as a $K$-algebra.
Further, $A/\nil(A)$ is a regular ring.
\end{cor}
\begin{proof}
The quotient $B := A/\nil(A)$ is a commutative, reduced $K$-algebra that is zero-product balanced by \autoref{prp:BalancedQuotient}.
It follows from \autoref{prp:SemiprimeCommutative} that $B$ is generated by idempotents.

To verify that $B$ is regular, let $x \in B$.
If $x = 0$, then $y= 0$ satisfies $x = xyx$.
Thus, we may assume that $x$ is nonzero, in which case we can choose pairwise orthogonal idempotents $e_1, \ldots, e_n \in B$ and nonzero, pairwise distinct $\lambda_j \in K$ such that $x = \sum_{j=1}^n \lambda_j e_j$.
Set $y := \sum_{j=1}^n \lambda_j^{-1} e_j$.
Then $x = xyx$, as desired.
\end{proof}

\begin{cor}
\label{prp:CommutativeSpanned}
Let $A$ be a commutative, zero-product balanced $K$-algebra over a field $K$, and let $\pi \colon A \to A/\nil(A)$ be the quotient map.
Then the following hold:
\begin{enumerate}\item 
There exists a unique $K$-algebra homomorphism $\sigma \colon A/\nil(A) \to A$ such that $\pi \circ \sigma = \id_{A/\nil(A)}$.
\item 
Every element of $A$ has the form
\[
x + \lambda_1 e_1 + \ldots + \lambda_n e_n
\]
for a nilpotent $x \in A$, pairwise orthogonal idempotents $e_1, \ldots, e_n \in A$, and nonzero, pairwise distinct $\lambda_1, \ldots, \lambda_n \in K$.
Moreover, this presentation is unique up to permutation of the indices.
\end{enumerate}
\end{cor}
\begin{proof}
(1). The quotient $A/\nil(A)$ is generated by idempotents by \autoref{prp:QuotientByNil}.
An ideal is said to be \emph{nil} if each of its elemenst is nilpotent.
It is a classical result that idempotents lift modulo nil ideals.
If the ring is even commutative, then the lift is unique;
see, for example \cite[Proposition~7.14]{Jac89BasicAlg2}.
This implies the existence and uniqueness of the map $\sigma$.

(2). Let $a \in A$.
Set
\[
\bar{a} := \pi(a) \in A/\nil(A) \andSep
x := a - \sigma(\bar{a}).
\]
We have $a = x + \sigma(\bar{a})$.
Since $\pi(x) = 0$, we have $x \in \nil(A)$, and this decomposition of $a$ as a sum of a nilpotent element and an element in the image of $\sigma$ is unique.
Further, using that $A/\nil(A)$ is generated by idempotents, and that $\sigma$ is $K$-linear and multiplicative, it follows that there exist pairwise orthogonal idempotents $e_1, \ldots, e_n \in A$, and nonzero, pairwise distinct $\lambda_1, \ldots, \lambda_n \in K$ such that $\sigma(\bar{a}) = \sum_{j=1}^n \lambda_j e_j$.
The uniqueness of the presentation of $\bar{a}$ as a linear combination of idempotents gives the same for the presentation of $\sigma(\bar{a})$.
\end{proof}

A ring $R$ is said to be \emph{semiregular} if $R/\rad(R)$ is regular, and every idempotent in $R/\rad(R)$ can be lifted to an idempotent in $R$.

Exchange rings were introduced in the unital setting by Warfield in \cite{War72ExchRgsDecMods}.
This was generalized to the nonunital case by Ara in \cite{Ara97ExtExch}:
A ring $R$ is an \emph{exchange ring} if for every $x \in R$ there exists an idempotent $e \in R$ and $r,s \in R$ such that $e = rx = x + s - sx$.
Every regular ring is an exchange ring.

Following Nicholson, \cite{Nic77LiftIdemExchRgs}, we say that a unital ring is \emph{clean} if every element is a sum of an invertible and an idempotent.
This was generalized by Nicholson-Zhu, \cite{NicZho05CleanGenRgs}, to nonunital rings.

\begin{cor}
\label{prp:Semiregular}
Let $A$ be a commutative, zero-product balanced $K$-algebra over a field $K$.
Then $A$ is semiregular, clean, and an exchange ring.
\end{cor}
\begin{proof}
By \autoref{prp:Radicals}, we have $\nil(A) = \rad(A)$.
Since idempotents lift modulo nil ideals (\cite[Proposition~7.14]{Jac89BasicAlg2}), it follows that idempotents in $R/\rad(A)$ lift.
The quotient $B := A/\nil(A)$ is a regular ring by \autoref{prp:QuotientByNil}.
It follows by definition that $A$ is semiregular.

As noted in the examples before Proposition~1.3 in \cite{Ara97ExtExch}, the class of (not necessarily unital) exchange rings includes all radical rings and all regular rings.
Thus, $\nil(A)$ and $A/\nil(A)$ are exchange rings.
Since idempotents in $A/\nil(A)$ lift, it follows from \cite[Theorem~2.2]{Ara97ExtExch} that $A$ is an exchange ring. 
Finally, it follows from \cite[Theorem~2]{NicZho05CleanGenRgs} that a commutative ring is clean if and only if it is an exchange ring.
(The unital case was shown in \cite[Proposition~1.8]{Nic77LiftIdemExchRgs}.)
\end{proof}

\begin{rmk}
Commutative $K$-algebras that are generated by idempotents have been studied in different contexts.
For example, by \cite[Corollary~2.3]{Kaw02AlgGenIdemGpAlg}, if $K$ is an algebraically closed field of characteristic zero, and $G$ is a torsion, abelian group, then the group algebra $K[G]$ is generated by idempotents, and a unital, commutative $K$-algebra $A$ is generated by idempotents if and only if it there exists a surjective $K$-algebra homomorphism $K[H] \to A$ for some torsion group $H$.

A different viewpoint is through Boolean powers:
If $K$ is a field, and $B$ is a Boolean algebra with associated compact Boolean space $X$, then the algebra $C(X,K_d)$ of finite-valued, continuous functions $X \to K$ is called the \emph{Boolean power} 
of $K$ by $B$.
It follows from \cite[Theorem~2.7]{BezMarMorOlb15IdemGenAlgsBooleanPowers} that a unital $K$-algebra is generated by idempotents if and only if it is a Boolean power.
This also provides a proof of the equivalence between~(2) and~(3) in \autoref{prp:SemiprimeCommutative} in the unital case.
\end{rmk}

\section{A dichotomy for zero-product balanced algebras}
\label{sec:Dichotomoy}

In this section we use our findings for commutative, zero-product balanced algebras from \autoref{sec:Commutative} to obtain structural results for arbitrary zero-product balanced algebras.

\smallskip

By a \emph{character} on a $K$-algebra $A$ we mean a nonzero $K$-algebra homomorphism $A \to K$.
By \autoref{prp:Radicals}, if $A$ is a commutative, zero-product balanced $K$-algebra over a field~$K$, then its nilradical $\nil(A)$ agrees with its Jacobson radical $\rad(A)$.

\begin{prp}
\label{prp:DichotomyCommutative}
Let $A$ be a commutative, zero-product balanced $K$-algebra over a field~$K$. 
Then $A$ either admits a character or is (nil)radical.
If $A$ is also \emph{unital}, then it admits a character.
\end{prp}
\begin{proof}
It is clear that $A$ cannot admit a character and be nil-radical at the same time.
Assuming that $A$ is not nil-radical, we need to show that $A$ admits a character.
Then $A$ is reduced, and thus $A \cong C_c(X,K_d)$ for some Boolean space $X$, by \autoref{prp:SemiprimeCommutative}.
Evaluating at some point in $X$ defines a character on $C_c(X,K_d)$.

If $A$ is unital, then it cannot be nil-radical and therefore admits a character.
\end{proof}

\begin{thm}
\label{prp:DichotomyGeneral}
Let $A$ be a zero-product balanced $K$-algebra over a field~$K$.
Then either $A$ has a character, or $A$ is a radical extension over its commutator ideal, that is, for every $a\in A$ exist $m,n\geq 1$ and $r_j,x_j,y_j,s_j\in A$ such that
\[
a^m = \sum_{j=1}^n r_j [x_j,y_j] s_j.
\]
\end{thm}
\begin{proof}
We may assume that $A\neq\{0\}$.
Let $I$ denote the commutator ideal of $A$:
\[
I = \left\{ \sum_{j=1}^n r_j[x_j,y_j]s_j : r_j,x_j,y_j,s_j \in A \right\}.
\]

We first show that $A$ cannot admit a character and be a radical extension of~$I$ at the same time:
Let $\pi \colon A \to K$ be a character.
Since $K$ is commutative, we have $\pi([x,y])=0$ for every $x,y \in A$.
It follows that $\pi(I) = \{0\}$.
Given $a \in A$ and $m \geq 1$ with $a^m \in I$, we have $\pi(a)^m = \pi(a^m) = 0$, and thus $\pi(a) = 0$.
Thus, if $A$ is a radical extension of $I$, then $\pi(A) = \{0\}$, a contradiction.

The quotient algebra $A/I$ is a commutative $K$-algebra that is zero-product balanced by \autoref{prp:BalancedQuotient}.
If $A/I$ admits a character, then the composition $A \to A/I \to K$ is a character on $A$.
Thus, we may assume that $A/I$ has no character.
Then $A/I$ is nil-radical by \autoref{prp:DichotomyCommutative}, and it follows that $A$ is a radical extension of~$I$.
\end{proof}

We showed in \cite{GarThi23arX:GenByCommutators} that a unital $K$-algebra $A$ is generated by its commutators as an ideal if and only if there exists $N \in \NN$ such that every element in $A$ is a sum of~$N$ products of pairs of commutators.
If $A$ is also zero-product balanced, then every commutator of $A$ is a sum of orthogonally factorizable square-zero elements by \autoref{prp:BalancedCommutatorFN2}, and it follows that every element in $A$ is a sum of products of pairs of orthogonally factorizable square-zero elements.
However, it remains unclear if there is a uniform bound on the number of required summands, since \autoref{prp:BalancedCommutatorFN2} provides no bound on the number of summands needed to express a commutator as a sum of orthogonally factorizable square-zero elements;
see \autoref{qst:Bound}.

\begin{thm}
\label{prp:UnitalGeneratedCommutators}
Let $A$ be a unital, zero-product balanced $K$-algebra that is generated by its commutators as an ideal.
Then for every $a \in A$ there exist $n \in \NN$ and elements $v_j,w_j,x_j,y_j \in A$ such that
\[
a = \sum_{j=1}^n v_jw_jx_jy_j, \andSep w_jv_j=y_jx_j=0 \text{ for } j=1,\ldots,n.
\]
(Note that $v_jw_j$ and $x_jy_j$ are orthogonally factorizable square-zero elements.)
\end{thm}
\begin{proof}
By \cite[Theorem~3.4]{GarThi23arX:GenByCommutators}, we have $A = \linSpan_K [A,A]\cdot[A,A]$.
By \autoref{prp:BalancedCommutatorFN2}, we have $\linSpan_K [A,A] = \linSpan_K FN_2(A)$, and thus $A = \linSpan_K FN_2(A) \cdot FN_2(A)$, which implies the statement.
\end{proof}

\begin{qst}
\label{qst:Bound}
Let $A$ be a unital, zero-product balanced $K$-algebra that is generated by its commutators as an ideal.
Does there exist $N \in \NN$ such that every element in $A$ is a sum of $N$ products of pairs of elements in $FN_2(A)$?
\end{qst}

\begin{cor}
\label{prp:DichotomyUnitalStrong}
Let $A$ be a \emph{unital}, zero-product balanced $K$-algebra over a field~$K$.
Then the following are equivalent:
\begin{enumerate}
\item
$A$ admits no character;
\item
$A$ is generated by its nilpotent elements as an ideal;
\item
every element in $A$ is a sum of products of pairs of (orthogonally factorizable) square-zero elements.
\end{enumerate}
\end{cor}
\begin{proof}
Let $I$ denote the commutator ideal of $A$ as in the proof of \autoref{prp:DichotomyGeneral}.
If $I \neq A$, then $A/I$ is a unital, commutative $K$-algebra that is zero-product balanced by \autoref{prp:BalancedQuotient}.
Then, by \autoref{prp:DichotomyCommutative}, $A/I$ admits a character.
This shows that~(1) implies~(2).
By \autoref{prp:UnitalGeneratedCommutators}, (2) implies~(3).
Since characters map square-zero elements to zero, we see that~(3) implies~(1).
\end{proof}


\providecommand{\etalchar}[1]{$^{#1}$}
\providecommand{\bysame}{\leavevmode\hbox to3em{\hrulefill}\thinspace}
\providecommand{\noopsort}[1]{}
\providecommand{\mr}[1]{\href{http://www.ams.org/mathscinet-getitem?mr=#1}{MR~#1}}
\providecommand{\zbl}[1]{\href{http://www.zentralblatt-math.org/zmath/en/search/?q=an:#1}{Zbl~#1}}
\providecommand{\jfm}[1]{\href{http://www.emis.de/cgi-bin/JFM-item?#1}{JFM~#1}}
\providecommand{\arxiv}[1]{\href{http://www.arxiv.org/abs/#1}{arXiv~#1}}
\providecommand{\doi}[1]{\url{http://dx.doi.org/#1}}
\providecommand{\MR}{\relax\ifhmode\unskip\space\fi MR }
\providecommand{\MRhref}[2]{%
  \href{http://www.ams.org/mathscinet-getitem?mr=#1}{#2}
}
\providecommand{\href}[2]{#2}

\end{document}